\documentclass{article}
\usepackage[utf8]{inputenc}
\usepackage{amssymb}
\usepackage{amsmath}
\usepackage{enumerate}
\usepackage{amsmath,amssymb}
\usepackage[utf8]{inputenc}
\usepackage{fourier, heuristica}
\usepackage{mathtools, nccmath}
\usepackage{mathtools}
\usepackage[english]{babel}
\usepackage{amsthm}
\usepackage{systeme}
\usepackage{geometry}[margin=1in]
\usepackage{authblk}
\newtheorem{theorem}{Theorem}[section]
\newtheorem{corollary}{Corollary}[theorem]
\newtheorem{lemma}[theorem]{Lemma}
\newtheorem{definition}[theorem]{Definition}
\newtheorem{ques}[theorem]{Question}
\newtheorem{prop}[theorem]{Proposition}

\newtheorem*{theorem*}{Theorem}

\usepackage{graphicx}
\usepackage{relsize}
\usepackage{tcolorbox}
\usepackage{tikz-cd}
\usepackage[all]{xy}
\usepackage{url}
\numberwithin{equation}{section}
\usepackage{eqparbox}

\usepackage{hyperref, cleveref}

\usepackage{hyperref}
\hypersetup{
    colorlinks=true,
    linkcolor=blue,
    filecolor=magenta,      
    urlcolor=cyan,
    pdftitle={Overleaf Example},
    pdfpagemode=FullScreen,
    }

\makeatletter
\newcommand{\address}[1]{\gdef\@address{#1}}
\newcommand{\email}[1]{\gdef\@email{\url{#1}}}
\newcommand{\@endstuff}{\par\vspace{\baselineskip}\noindent\small
\begin{tabular}{@{}l}\scshape\@address\\\textit{E-mail address:} \@email\end{tabular}}
\AtEndDocument{\@endstuff}
\makeatother
\title{Separation of horocycle orbits on moduli space in genus 2}
\author{John Rached}
\date{}
\address{Department of Mathematics and Statistics,
Binghamton University. Binghamton, New York, United States}
\email{jabourached@math.binghmaton.edu}

\begin{document}
\maketitle

\begin{abstract}
We prove a quantitative closing lemma for the horocycle flow induced by the $\mathrm{SL}(2,\mathbb{R})$-action on the moduli space of Abelian differentials with a double-order zero on surfaces of genus 2. The proof proceeds via construction of a Margulis function measuring the discretized fractal dimension of separation of a horocycle orbit of a point from itself, in a direction transverse to the $\mathrm{SL}(2,\mathbb{R})$-orbit. From this, we deduce that small transversal separation guarantees the existence of a nearby point with a pseudo-Anosov in its Veech group. This is reminiscent of the initial dimension phases in Bourgain-Gamburd for random walks on compact groups, Bourgain-Lindenstrauss-Furman-Mozes for quantitative equidistribution in tori, and quantitative equidistribution of horocycle flow for a product of $\mathrm{SL}(2,\mathbb{R})$ with itself due to Lindenstrauss-Mohammadi-Wang, and multiple other works. 
\end{abstract}

\section{Introduction}

Dynamical closing lemmas have been investigated for several decades, in a broad variety of contexts. In general terms, "closing lemma" refers to perturbing an initial point in a dynamical system, whose orbit is recurrent to a compact set, to find a point of small distance from the initial point that lies on a periodic orbit. "Quantitative closing lemma" generally indicates a bound in terms of the time parameter on the period of the closed orbit achieved. Two results that have inspired substantial research activity are Pugh's closing lemma for $C^1$-diffeomorphisms of compact smooth manifolds \cite{fc49ac35-5adf-3bc9-b971-cec3f7c4e0a5}, and Anosov's closing lemma \cite{katok1995introduction} for hyperbolic sets of a flow. An open line of research in the direction of the former is to achieve such a closing lemma for higher orders of smoothness than $C^{1}$ \cite{article}. This was resolved for the class of Hamiltonian diffeomorphisms of closed surfaces by Asaoka and Irie \cite{article2}. A quantitative result in this spirit, for Reeb orbits on contact 3-manifolds, was recently shown by Hutchings \cite{hutchings2024elementary}.\\

In Teichm{\"u}ller dynamics, the diagonal flow induced by the $\mathrm{SL}(2,\mathbb{R})$-action on the Hodge bundle over moduli space projects to geodesics for the Teichm{\"u}ller metric on moduli space. This action is uniformly hyperbolic (Anosov) on compact sets; Hamendst{\"a}dt \cite{article3} and Eskin-Mirzakhani-Rafi \cite{eskin2019counting} have independently proven closing lemmas for this flow. A version of this lemma for affine invariant submanifolds of strata in the Hodge bundle is used by Wright in \cite{wright2014field}, to prove that such manifolds are defined over number fields. A similar version of this lemma is invoked in \cite{Filip2014ZeroLE} to show that the Zariski closure of the monodromy of the Kontsevich-Zorich cocycle restricted to the tangent bundle of an affine invariant manifold is the full group of endomorphisms preserving the symplectic form on the manifold. \\

Our concern in this article is with the action of a different one-parameter subgroup of $\mathrm{SL}(2,\mathbb{R})$ acting on a stratum of the Hodge bundle, namely, the action of the unipotent $U = \left\{u_s \ | \ s \in \mathbb{R}\right\}$, where $u_s = \begin{pmatrix}
1 & s \\
0 & 1
\end{pmatrix}$. The $u_s$-action remains poorly understood, in contrast to the full $\mathrm{SL}(2,\mathbb{R})$-action, where a complete classification of invariant and stationary measures was obtained by Eskin-Mirzkakhani \cite{eskin2018invariant}. Full $\mathrm{SL}(2,\mathbb{R})$-orbit closures are known to be affine invariant submanifolds of a stratum of the Hodge bundle \cite{eskin2015isolation}, and are in fact algebraic varieties \cite{filip2016splitting}. On the other hand, $u_s$-orbit closures can have fractional Hausdorff dimensions \cite{Chaika2020TremorsAH} and exhibit other features distinctive from the homogeneous theory \cite{chaika2023space}. Nevertheless, we show that in the single double-zero setting in genus 2, small transversal separation of a geodesic push of a horocycle orbit guarantees proximity to a periodic orbit, establishing a resonance in moduli space for the "initial dimension phase" in Bourgain-Furman-Lindenstrauss-Mozes \cite{bourgain2011stationary}, Lindenstrauss-Margulis \cite{lindenstrauss2014effective}, Yang \cite{Yang} and Lindenstrauss- \\ Mohammadi-Wang \cite{lindenstrauss2022effectiveunipotent}. The dichotomy between large dimension of transversal separation and nearly lying on a periodic orbit was used in the aforementioned works to establish quantitative equidistribution results in those settings; this is our original motivation as well.

\subsection{Overview}

For $G$ a connected Lie group, and $\Gamma$ a lattice, Ratner's seminal theorems \cite{Ratner}, \cite{Ratner2}, \cite{Ratner3} provide a classification for orbits of all points in $G/\Gamma$ under the action of a subgroup of $G$ generated by unipotent elements. Ratner's theorems rely on the pointwise ergodic theorem, making these results difficult to effectivize in most cases. Green and Tao prove equidistribution for nilflows with polynomial error rates \cite{green2012quantitative}. When the unipotent subgroup is horospherical, as in the case of the horocycle flow on $\mathrm{SL}(2,\mathbb{R})/\mathrm{SL}(2,\mathbb{Z})$, effective equidistribution for long unipotent orbits is known, and in the case of $\mathrm{SL}(2,\mathbb{R})/\mathrm{SL}(2,\mathbb{Z})$ can be deduced from effective ergodicity of the horocycle flow as discussed in an original work of Ratner \cite{Ratner4}. Recently, there has been activity in search of effective accounts of equidistribution of both flows and closed unipotent orbits beyond the horospherical and nilflow settings. One starting point of this was with the work of Einsiedler, Margulis and Venkatesh who achieve effective equidistribution of closed unipotent orbits with polynomial error rate in a general semisimple homogeneous (arithmetic) setting \cite{einsiedler2009effective}. The case of effective equidistribution with polynomial error rates for unipotent orbits on arithmetic quotients $\mathrm{SL}(2,\mathbb{R}) \times \mathrm{SL}(2,\mathbb{R})$ was also done recently by Lindenstrauss, Mohammadi and Wang \cite{lindenstrauss2022polynomial}. Using different methods, Str$\"{o}$mbergsson obtained polynomial effective equidistribution for a nonhorospherical action on the quotient of $\mathrm{SL}(2,\mathbb{R})\ltimes \mathbb{R}^2$ by $\mathrm{SL}(2,\mathbb{Z}) \ltimes \mathbb{Z}^2$ \cite{strombergsson2015effective}. \\

The results above leverage arithmeticity of the lattice $\Gamma$ in a crucial way. The overarching idea is to establish analogs of Liouville's theorem, which states roughly that irrational algebraic numbers are badly approximable by rational numbers. We provide some details to illustrate this concept at work in the setting of Bourgain-Furman-Lindenstrauss-Mozes. Let $\Gamma$ be a semigroup of $d \times d$ non-singular integer matrices. If the natural action on the torus $\mathbb{T}^d$ is strongly irreducible, then the action of $\Gamma$ is ergodic with respect to Haar measure. In the case $d=1$ and $\Gamma$ is Abelian, Furstenberg \cite{Furstenberg1} showed that provided $\Gamma$ contains no finite index cyclic groups, $\Gamma \cdot x$ is dense for all $x \in \mathbb{R} \setminus \mathbb{Q}$.  Bourgain, Furman, Lindenstrauss and Mozes \cite{bourgain2011stationary} considered the situation, where, along with other technical assumptions, the natural action of $\Gamma$ on $\mathbb{R}^d$ is assumed to be strongly irreducible (this is strictly stronger than the assumption that $\Gamma$ acting on $\mathbb{T}^d$ is strongly irreducible). When $\Gamma$ acts strongly irreducibly on $\mathbb{R}^d$, and $\nu$ is a probability measure on $\Gamma$, the top Lyapunov exponent $\lambda_1(\nu) = \lim\limits_{n \mapsto \infty} \frac{1}{n}\mathrm{log} \ \lvert\lvert g_1g_2 \cdot \cdot \cdot g_n \rvert\rvert$ is $\nu$-almost surely positive. Given a probability measure $\nu$ on $\Gamma$ and a probability measure $\mu$ on $\mathbb{T}^d$, one can define a probability measure on $\mathbb{T}^d$ by $\nu \ast \mu = \sum\limits_{g \in \Gamma} \nu(g) g_{*} \mu$. Under some additional mild assumptions, including a moment condition on $\nu$, Bourgain-Furman-Lindenstrauss-Mozes \cite{bourgain2011stationary} prove that if $\nu$ is a fixed probability measure supported on $\Gamma < \mathrm{SL}_d(\mathbb{R})$, then for any $0 < \lambda < \lambda_1$, there is a constant $C = C(\nu,\lambda)$, so that if for a point $x \in \mathbb{T}^d$ the measure $\mu_n = \nu^{(n)}\ast \delta_x$ satisfies that for some $a \in \mathbb{Z}^d \setminus \left\{0\right\}$

\begin{displaymath}
\bigl \lvert \hat{\mu}_n(a) > t > 0 \bigr \rvert \ \ \textrm{with} \ \ n > C \cdot \mathrm{log}\frac{2\lvert\lvert a \rvert \rvert}{t},
\end{displaymath}
then $x$ admits a rational approximation $\frac{p}{q}$ with $p \in \mathbb{Z}^d$ and $q \in \mathbb{Z}^{+}$ satisfying

\begin{displaymath}
    \Bigl \lvert \Bigl \lvert x - \frac{p}{q} \Bigr \rvert \Bigr \rvert < e^{-\lambda n} \ \ \mathrm{and} \ \lvert q \rvert < \left(\frac{2\lvert\lvert a \rvert \rvert}{t}\right)^C.
\end{displaymath}
A key step in the proof of this theorem is to start with an absolute lower bound on a single Fourier coefficient of the measure $\mu_n = \nu^{\ast(n)}\ast \delta_x$ and further conclude for a chosen $m < n$, $\mu_{n-m}$ has a "rich" set of Fourier coefficients larger than a polynomial in $t$. Starting with some $a_0 \in \mathbb{Z}^d \big\backslash \left\{0\right\}$ with $\lvert \hat{\mu}_n(a_0) > t_0$ for large enough $n$, the idea is to prove that for $t = t_0^p$, the set 

\begin{displaymath}
A_{n-m_1,t} = \left\{a \in \mathbb{Z}^d: \ \bigl \lvert \hat{\mu}_{n-m_1}(a) \bigr \rvert > t \right\}
\end{displaymath}
is a thick set in $\mathbb{Z}^d$. For a subset $E \subset \mathbb{Z}^d$, let $\mathcal{N}(E;R)$ denote the covering number of $E$ by balls of radius $R$. The following is the crucial step.

\begin{lemma}[B-F-L-M \cite{bourgain2011stationary}]
\label{lemma1.1}
There exists a large $N$ and an exponentially smaller $R$, so that the number of $R$-balls needed to cover the intersection of a rich set of Fourier coefficients with a large box is greater than a polynomial. Namely, $\mathcal{N} \left( A_{m-n_1,t_0^{p}} \cap [-N,N]^{d} ; R\right) > t_0^{p}\left(N/R\right)^{d}$.
\end{lemma}
By an analog of Wiener's lemma, Lemma $\ref{lemma1.1}$, gives a "small-dimensional" separation of points on the torus. This information is bootstrapped to reach a dimension of separation equal to the dimension of the torus.\\

Returning to the Lie group setting, this general scheme of achieving a small initial dimension of separation and boostrapping the dimension is successfully employed by Lindenstrauss-Mohammadi and Lindenstrauss-Mohammadi-Wang to achieve polynomially effective density and polynomially effective equidistribution of one-parameter unipotent flows in arithmetic quotients of $\mathrm{SL}(2,\mathbb{C})$ and $\mathrm{SL}(2,\mathbb{R}) \times \mathrm{SL}(2,\mathbb{R})$, and polynomial rates for the Oppenheim conjecture \cite{lindenstrauss2022effectiveunipotent}, \cite{lindenstrauss2022polynomial1}, \cite{lindenstrauss2023effective}. In these works, the authors prove quantitative closing lemmas (for instance, Proposition 4.8 in \cite{lindenstrauss2022effectiveunipotent}) which are themselves similar to Lemma 5.2 in Lindenstrauss-Margulis \cite{lindenstrauss2014effective}. In \cite{lindenstrauss2022effectiveunipotent}, \cite{lindenstrauss2022polynomial1}, \cite{lindenstrauss2023effective}, these closing lemmas are used to produce a small transversal dimension of separation in the unipotent direction, for unipotent orbits not too close to a periodic orbit. This serves as the analogue of Lemma \ref{lemma1.1}, and this dimension is bootstrapped to achieve dimension close to that of a horospherical subgroup, for which there exist methods to deduce effective density and effective equidistribution. We proceed to state our main theorem, which is a direct analog of Proposition 4.8 in \cite{lindenstrauss2022effectiveunipotent}, a direct analog of Proposition 13.1 in \cite{einsiedler2009effective} and similar in spirit to Lemma $\ref{lemma1.1}$. \\

Let $S$ = $S_{g}$ be a surface of genus $g$, and let $\mathcal{M}_{g}$ be the moduli space of Riemann surfaces homeomorphic to $S$. Let $\Omega\mathcal{M}_{g}$ be the space of Abelian differentials (holomorphic 1-forms $\omega$). Let $\Omega_1\mathcal{M}_g \subset \Omega\mathcal{M}_g$ be the unit-area locus of $\Omega\mathcal{M}_g$. There is an $\mathrm{SL}(2,\mathbb{R})$-action on $\Omega_1\mathcal{M}_{g}$ given by the action of $\mathrm{SL}(2,\mathbb{R})$ on $\mathbb{R}^2$ corresponding to "multiplication by $A \in \mathrm{SL}(2,\mathbb{R})$" on the atlas of charts to $\mathbb{R}^2$ determined by the quadratic differential. By Teichm{\"u}ller theory, the orbits of the diagonal flow $a_t = \begin{bmatrix}
e^t & 0\\
0 & e^{-t}
\end{bmatrix}$ project to Teichm{\"u}ller geodesics under the natural projection map $\pi:\Omega_1\mathcal{M}_{g} \rightarrow \mathcal{M}_{g}$.  $\Omega\mathcal{M}_{g}$ has a natural stratification: we say $(X,\omega) \in \Omega\mathcal{M}_{g}$ is of type $ \sigma =\left(p_1,...,p_k\right)$ if $\omega$ has zeroes of order $\left\{p_i\right\}$. The space of all Abelian differentials of type $\sigma = \left(p_1,...,p_k\right)$ in $\Omega\mathcal{M}_{g}$ will be denoted by $\Omega\mathcal{M}_{g}\left(\sigma\right)$. This is a complex-analytic orbifold of real dimension $4g + 2k - 1$. It also has the structure of a complex algebraic variety. Let $\mathcal{T}_g$ be the Teichm{\"u}ller space, the orbifold universal cover of $\mathcal{M}_g$. Let $\Omega\mathcal{T}_g$ and $\Omega_1\mathcal{T}_g$ be defined similarly as above. We have an orbifold covering map:

\begin{displaymath}
\pi:\Omega_1\mathcal{T}_g\left(\sigma\right) \rightarrow \Omega_1\mathcal{M}_g\left(\sigma\right).
\end{displaymath}
Moreover, if $V \subset \Omega_1\mathcal{M}_g\left(\sigma\right)$ is a properly immersed manifold, the pre-image $\pi^{-1}(V)$ is a properly immersed manifold of the same dimension. The lowest dimensional $\mathrm{SL}(2,\mathbb{R})$-invariant manifolds that can be cut out locally by real-linear equations are of particular relevance to us. Specifically, we are interested in the orbit closures of $\pi\left(\mathrm{SL}(2,\mathbb{R})\cdot \omega\right)$ for special choices of $(X,\omega) \in \Omega_1\mathcal{T}_g\left(\sigma\right)$. For $\mathrm{Mod}(S)$ the mapping class group of $S = S_g$, it follows from a classical result of of Smillie-Weiss \cite{smillie2004minimal} that if $\Gamma_{\omega} \leq \mathrm{Mod}(S)$ is the set-wise stabilizer of the lift of $\pi\left(\mathrm{SL}(2,\mathbb{R})\cdot \omega\right)$ to the connected component $\widetilde{{\pi^{-1} \circ {\pi\left(\mathrm{SL}(2,\mathbb{R})\cdot \omega\right)}}}$ of
$\Omega_1\mathcal{T}_g\left(\sigma\right)$ containing $(X,\omega)$, then

\begin{displaymath}
\widetilde{{\pi^{-1}\circ{\pi\left(\mathrm{SL}(2,\mathbb{R})\cdot \omega\right)}}} =  \Gamma_w \cdot \left(\overline{\mathrm{SL}(2,\mathbb{R})\cdot \omega}\right)
\end{displaymath}
when $(X,\omega)$ has lattice-stabilizer. In this case, $(X,\omega)$ is referred to as a $\textit{Veech}$ surface. Veech surfaces have special dynamical properties. In particular, for such $\omega$, $\pi\left(\mathrm{SL}(2,\mathbb{R})\cdot \omega\right)$ is a closed orbit. Furthermore, the image of $\mathrm{SL}(2,\mathbb{R})\cdot \omega$ under

\begin{displaymath}
\Omega_1\mathcal{T}_g\left(\sigma\right) \xrightarrow{\pi} \Omega_1\mathcal{M}_g\left(\sigma\right) \rightarrow \mathcal{M}_g
\end{displaymath}
is an algebraic curve $V$ in $\mathcal{M}_g(S)$, called a Teichm{\"u}ller curve. Let $e^{-0.01t} < \beta = e^{-\kappa t} < 1$ for a fixed $\kappa = 1/D > 0$  where $D$ will be specified below. For every $t \geq 0$, define

\begin{displaymath}
E_{t} = B_{\beta} \cdot a_{t} \cdot \left\{ u_r: r \in [0,1] \right\} \subset \mathrm{SL}(2,\mathbb{R})
\end{displaymath}
where $B_{\beta}:= \left\{u_s^{\intercal}: |s| \leq \beta \right\} \cdot \left\{a_t: |t| \leq \beta \right\}$ and $u_s^{\intercal}$ is the transpose of $u_s$. In words, this is a "smearing" by a small amount of $\mathrm{SL}(2,\mathbb{R})$ of a geodesic push of length $t$ of a length-1 horocycle segment. The principal reason for stating our separation theorem with a smearing approximation of geodesic translates of unit-length horocycle orbits is for potential application to quantitative equidistribution theorems. For a concise discussion of this, see Section 6 of $\cite{lindenstrauss2022polynomial}$.\\

The following Theorem shows that, up to removing an exceptional set of very small measure, geodesic pushes of length-1 horocycle segments are separated in a direction tranverse to the $\mathrm{SL}(2,\mathbb{R})$-orbit of a point (up to "smearing"), so long as the orbit of this point doesn't come too close to a Teichm{\"u}ller curve, with a quantitative relationship between the proximity to the Teichm{\"u}ller curve and the size of a hyperbolic element the curve contains in its Veech group. We state our Main Theorem.

\begin{theorem}
\label{main_theorem}
    Let $\Omega_1\mathcal{M}_2(2)$ be the stratum of unit-area Abelian differentials of genus 2 with a single zero. There exist $D_0 \geq 0, \alpha$ depending only on the stratum, satisfying the following. Suppose $D \geq D_0 + 1$. For any $x = (X,\omega) \in \Omega_1\mathcal{M}_2(2)$, there exists $t'$ depending only on the injectivity radius of $x$ such that for all $t \geq t'$, there exists a function $f_t: E_{t}\cdot x \rightarrow (0,\infty)$ so that at least one of the following is true.

    \begin{enumerate}
    \item There is a subset $I \subset [0,1] $ with $  \bigl \lvert [0,1] \big \backslash I \bigr \rvert \ll \beta^{\alpha}$ such that for all $s \in I$ we have the following
    \begin{enumerate}
    \item $\mathrm{inj}(a_{8t}u_sx) \geq \beta$,
    \item $h \rightarrow h \cdot a_{8t} u_s x$ is an injective map $E_t \mapsto \Omega_1\mathcal{M}_2(2)$,
    \item for all $z \in E_t \cdot a_{8t}u_sx$, we have
    \begin{displaymath}
    f_{t}(z) \leq e^{Dt}.
    \end{displaymath}
    \end{enumerate}
    \item There is $x'=(X',\omega') \in \Omega_1\mathcal{M}_2(2)$ such that $x'$ lies on a Teichm{\"u}ller curve $V$ and \\ $d_{\mathcal{T}}\left(z',x'\right) \leq e^{(-D+D_0)t}$ for some $z' \in a_{8t}u_{s'}\cdot x$, where $d_{\mathcal{T}}(\cdot,\cdot)$ is the Teichm{\"u}ller distance and $s' \in [0,1]$. The Veech group of $x'$ contains a hyperbolic element $\gamma$ with $\mathrm{log}(l(\gamma)) \leq e^{Dt}$, where $l(\gamma)$ is the translation length.
    \end{enumerate}
\end{theorem}

We expect Theorem $\ref{main_theorem}$ to fail for all other strata $\Omega_1\mathcal{M}_g(\sigma)$, though we supply no proof in this article. Essentially, this is due to the existence of non-lattice Veech groups in other strata which are infinitely generated, constructed by Hubert and Schmidt \cite{Veech_groups}, and McMullen \cite{infinite_complexity_McMullen}. However, we make the following conjecture, which we believe would serve as a substitute for the "initial dimension phase" for all strata of moduli space, opening the door to adopting the schema described above to establish quantitative equidistribution results for horocycle flow in all strata. In particular, combining Theorem $\ref{main_theorem}$ with results of Sanchez on effective equidistribution of large dimensional measures \cite{sanchez2023effective} would leave only the 'bootstrapping' step of \cite{lindenstrauss2022effectiveunipotent}. Since \cite{sanchez2023effective} applies to all strata, a similar obstacle would remain were the following conjecture to be established.

\begin{ques}
\label{crit_exp}
Does there exist a sequence of Veech groups \ $\Gamma_k \in \Omega_1\mathcal{M}_g(\sigma)$ in any stratum with critical exponents approaching one?
\end{ques}

An affirmative answer to the analogous statement for Kleinian groups has recently been established by \cite{solan2024critical}, with powerful techniques that may be applicable to Teichm{\"u}ller dynamics. \\

The paper is organized as follows. In Section $\ref{prelims}$, we review fundamental facts on the geometry of moduli space, which we rely on to adapt the proofs of homogeneous closing lemmas to this context. In Section \ref{Section3}, we use results of Minsky-Weiss to state quantitative-nondivergence that establishes enough recurrence to compact sets to reduce the statement of Theorem $\ref{main_theorem}$ to a dichotomy. In Section $\ref{Section4}$, we construct a Margulis function that measures the transversal (to the $\mathrm{SL}(2,\mathbb{R})$ direction) separation of an orbit from itself, and show the group of elements in $\mathrm{SL}(2,\mathbb{R})$ bringing an orbit close back to itself in the tranverse direction cannot be unipotent. We conclude there must be a hyperbolic element of $\mathrm{SL}(2,\mathbb{R})$ stabilizing a nearby element, and since we are working in $\Omega_1\mathcal{M}_2(2)$, this element must lie on a Teichm{\"u}ller curve.

\section*{Acknowledgements}
The author thanks Asaf Katz for helpful discussions on homogeneous dynamics, and Alex Eskin, Paul Apisa, Jon Chaika, Alex Wright and Jenya Sapir for insightful comments.

\section{Preliminaries}
\label{prelims}

\subsection{Constants and matrix norms}
\label{subsection_matrix_norm}

We will use the notation $B \ll C$ to mean $B \leq kC$, where $k$ is only allowed to depend on the stratum. For a matrix $M \in \mathrm{SL}(2,\mathbb{R})$, we will denote by $||\cdot||$ the Frobenius norm. Recall the decomposition $ M = KAK$ where $A$ is the diagonal group and $K = SO(2)$ is the maximal compact rotation group. The $KAK$ decomposition coincides with the singular value decomposition $M = U \Sigma V^{\perp}$ up to changing signs of the diagonal entries of the matrices in the middle. We make extensive use of the following elementary facts. For $M_1,M_2 \in \mathrm{SL}(2,\mathbb{R})$, and using $\sigma_{min}(M)$ to denote the smallest singular value of a matrix:

\begin{enumerate}
\item $||M_1|| = ||M^{-1}_1||$
\item $||M_1M_2|| \leq ||M_1||\cdot ||M_2||$
\item $||M_1M_2|| \geq \sigma_{min}(M_1) \cdot ||M_2||$.
\end{enumerate}

\subsection{Teichm{\"u}ller metric and Teichm{\"u}ller curves}
\label{Section2.1}

In this subsection, we briefly review the metric structures we will be interested in, and the dynamics of Veech surfaces. All of the material here is standard, and we will sometimes follow $\cite{mcmullen2003billiards}$.
Let $S_{g}$ be a surface of genus $g$, satisfying $3g-3 > 0$. Let $\phi_1:S_g \rightarrow S_1$, $\phi_2:S_g \rightarrow S_2$, be two diffeomorphisms, where $S_1$ and $S_2$ have finite area. Denote by $\sim$ the equivalence relation
where $\phi_1 \sim \phi_2$ if $\phi_1 \circ \phi^{-1}_2$ is isotopic to a holomorphic diffeomorphism. Define

\begin{displaymath}
\mathcal{T}_{g} = \big\{\phi:S_{g} \rightarrow S \ | \ \text{S is finite area }, \phi \ \text{a diffeomorphism}\big\} \big/ \sim.
\end{displaymath}
Given $\psi \in \text{Diff}^{+}\left(S_{g}\right)$, the group of orientation-preserving diffeomorphisms of $S_{g}$, we obtain a map $\mathcal{T}_{g} \rightarrow \mathcal{T}_{g}$ by precomposition by $\psi$. By the definition of the equivalence
relation $\sim$, if $\psi$ is contained in the normal subgroup $\text{Diff}^{0}\left(S_{g}\right)$ of diffeomorphisms isotopic to a holomorphic diffeomorphism, this map defined by precomposition is trivial. Therefore, we have an action of the mapping class group of $S_{g}$ 

\begin{displaymath}
\text{Mod}\left(S_{g} \right):= \text{Diff}^{+}\left(S_{g}\right) \big/ \text{Diff}^{0}\left(S_{g}\right)
\end{displaymath}
on $\mathcal{T}_{g}(S)$. The moduli space $\mathcal{M}_{g}$ of $S_{g}$ is defined as 

\begin{displaymath}
\mathcal{M}_{g}:= \mathcal{T}_{g}\big/ \text{Mod}\left(S_{g}\right).
\end{displaymath}

$\mathcal{T}_{g}$ is a real-analytic space, and is homeomorphic to $\mathbb{R}^{6g-6}$. The cotangent space at $X \in \mathcal{T}_{g}$ can be identified with $Q(X)$, the space of holomorphic quadratic differentials on $X$. The dual space (tangent space at $X$) is naturally identified with the space of harmonic Beltrami differentials on $X$, ${B}(X)$. Suppose $\mu = \mu(z)d\bar{z}\otimes(dz)^{-1} \in B(X)$ and $q = q(z)dz^2 \in Q(X)$. There is a natural non-degenerate pairing on $\mathcal{T}_g$
\begin{displaymath}
\langle{\mu,q\rangle} = \int_X \mu(z)q(z)|dz|^2.
\end{displaymath}
The norm $||\mu||_{\mathcal{T}} = \mathrm{sup}\left\{|\langle\{\mu,q\rangle\}:q \in Q(X), ||q||_1 = \int_X|q| = 1\right\}$ gives the $\textit{Teichm{\"u}ller metric}$ \ on $\mathcal{T}_g$. This is the infinitesimal description of the Teichm{\"u}ller metric. The distance between two points in Teichm{\"u}ller space has a simple description in terms of dilatation. We have

\begin{displaymath}
d_{\mathcal{T}}\left(\left(\phi_1:S_g \rightarrow S_1\right),\left(\phi_2:S_g \rightarrow S_2\right)\right) = \frac{1}{2}\underset{\phi}{\mathrm{inf}} \; \mathrm{log} \; K_{\phi}
\end{displaymath}
where $\phi: X_1 \rightarrow X_2$ ranges over all quasi-conformal maps from $X_1$ to $X_2$ isotopic to $\phi_1\circ\phi_2^{-1}$ and $K_{\phi} \geq 1$ is the dilatation coefficient. Since the definition of the metric only depends on the holomorphic structure at $X$, the Teichm{\"u}ller metric on $\mathcal{T}_g$ descends to a metric on $\mathcal{M}_{g}$.  \\

It is well-known that a pair $(X,q)$ with $ q \in Q(X)$ and $q \neq 0$, generates a holomorphic embedding from the hyperbolic plane $\mathbb{H}$ with Kobayashi metric on $\mathbb{H}$ and Teichm{\"u}ller metric on $\mathcal{T}_g$

\begin{displaymath}
\tilde{f}:\mathbb{H} \rightarrow \mathcal{T}_g.
\end{displaymath}
$\tilde{f}$ is an isometry. Passing to $\mathcal{M}_g = \mathcal{T}_g \big/ \mathrm{Mod}(S_g)$, we obtain a $\textit{complex geodesic}$

\begin{displaymath}
f: \mathbb{H} \rightarrow \mathcal{M}_g.
\end{displaymath}
We summarize the constructions of $\tilde{f}$ and $f$, and the connection to dynamics. For $t$ a complex parameter, the Riemann surface $X_t = \tilde{f}(t)$ is characterized by the property that the complex dilatation $\mu_t$ of the extremal quasi-conformal map $\psi_t:X \rightarrow X_t$ is given by

\begin{equation}
\label{equation2.1}
\mu_t = \left(\frac{i-t}{i+t}\right) \cdot \frac{\bar{q}}{|q|}.
\end{equation}
Now, assume $q$ is the square of a holomorphic $1$-form $w$, $q = \omega^2$. Then, these complex geodesics have an interpretation in terms of the $\mathrm{SL}(2,\mathbb{R})$ action on $\Omega_1\mathcal{T}_g\left(\sigma\right)$. In particular, any holomorphic $1$ -form $\omega$ yields, away from its zeroes, a flat metric Euclidean metric $|\omega|$ and an atlas of charts $U_i \rightarrow \mathbb{C}$ whose transition functions are translations. The $\mathrm{SL}(2,\mathbb{R})$-action on $\Omega_1\mathcal{T}_g\left(\sigma\right)$ is defined by composition of these charts with matrices $A \in \mathrm{SL}(2,\mathbb{R})$ acting linearly on $\mathbb{C} \simeq \mathbb{R}^2$. $A \in \mathrm{SL}(2,\mathbb{R})$ can be interpreted as an affine map in these coordinate charts $A\cdot(X,\omega) = (X',\omega')$. This yields a quasi-conformal map $f: X \rightarrow X'$ with terminal holomorphic $1$-form $\omega'$. In the case that $A \in \mathrm{SO}(2,\mathbb{R})$, the map $f$ is conformal, so the action descends to a faithful action of $\mathrm{SL}(2,\mathbb{R})/\mathrm{SO}(2,\mathbb{R})$. One can show that this descended action induces an isometric injection

\begin{displaymath}
\mathrm{SL}(2,\mathbb{R})/\mathrm{SO}(2,\mathbb{R})(X,\omega) \rightarrow \mathcal{T}_g.
\end{displaymath}
Furthermore, composition with the projection $\pi:\mathcal{T}_g \rightarrow \mathcal{M}_g$ defines the complex geodesic

\begin{displaymath}
f: \mathbb{H} \rightarrow \mathcal{M}_g.
\end{displaymath}
One can check that if $t \in \mathbb{C}$ parametrizes the Riemann surfaces, the complex dilatation of the extremal quasi-conformal map $\psi_t: X \rightarrow X_t$ is given by

\begin{equation}
\mu_t = \left(\frac{i-t}{i+t}\right) \cdot \frac{\bar{\omega}}{|\omega|}
\end{equation}
which accords with Equation $\ref{equation2.1}$. Furthermore, $f$ factors through the quotient space $V = \mathbb{H} / \Gamma$, where $\Gamma = \left\{B \in \mathrm{Aut}(\mathbb{H}): f(B\cdot t) = f(t) \ \forall t \right\}$. In the case that $\Gamma$ is a lattice, we call the quotient map

\begin{displaymath}
f:V \rightarrow \mathcal{M}_g
\end{displaymath}
a Teichm{\"u}ller curve. $f$ is proper and generically injective. The image $f(V)$ is not a normal subvariety, but we will not distinguish between the Teichm{\"u}ller curve as a map $f$ and the normalization of $f(V)$. One can interpret $\Gamma$ above directly in terms of flat geometry. Let $\mathrm{Aff}^{+}(X,\omega)$ be the set of orientation-preserving affine diffeomorphisms of $(X,\omega)$. For such an affine diffeomorphism $\phi:(X,\omega) \mapsto (X,\omega)$, the derivative $D\phi$ is constant. Since these diffeomorphisms are assumed to be orientation preserving, and also preserve the area of $\omega$, the map $\phi \mapsto D\phi$ defines a map of $\mathrm{Aff}^{+}(X,\omega)$ into $\mathrm{SL}(2,\mathbb{R})$. We will call the image of this map the $\textit{Veech}$ group of $(X,\omega)$, and denote it by $\mathrm{SL}(X,\omega)$. By Proposition 3.2 in \cite{mcmullen2003billiards}, $\Gamma$ and $\mathrm{SL}(X,\omega)$ are conjugate.

\begin{subsection}{Period Coordinates and Holonomy Vectors}

For any $(X,\omega) \in \Omega\mathcal{M}_g\left(\sigma\right)$, there exists a triangulation $T$ of the underlying surface $X$ by saddle connections. Denote the zeroes of $\omega$ by $Z(\omega)$, where $|Z(\omega)|=k$. One can choose $h = 2g + |Z(\omega)| - 1$ directed edges $\left\{w_i\right\}_{i=1}^h$ of $T$, and an open neighborhood $U \subset \Omega\mathcal{M}_g(\sigma)$ of $(X,\omega)$ in its stratum, so that there exists an open analytic embedding:

\begin{displaymath}
\phi_{T,\omega}: U \rightarrow H^1(X,Z(\omega),\mathbb{C}) \cong \mathbb{C}^{2g + Z(\omega) -1}
\end{displaymath}
called the period map. It is defined by $(X,\omega)$ to be the relative class $[\omega] \in H^1(X,Z(\omega),\mathbb{C})$ which satisfies
$\left<[\omega],w_i\right> = \int_{w_i}\omega$ for all $w_i$ viewed as relative classes in $H_1(X,Z(\omega),\mathbb{C})$. Thus, we may think of $[\omega]$ as a local coordinate on $\Omega \mathcal{M}_g(\sigma).$ For any other geodesic triangulation $T'$, the map $\phi_{T',\omega}\circ \phi^{-1}_{T,\omega}$ is linear.
A consequence of the uniformization theorem is that every element $(X,\omega) \in \Omega\mathcal{M}_g\left(\sigma\right)$ can be presented in the form

\begin{displaymath}
(X,\omega) = (P,dz)\big/ \sim
\end{displaymath}
for a polygon $P \subset \mathbb{C}$, and the $1$-form $dz$ on $\mathbb{C}$. The reason for this is that one can construct a geodesic triangulation of the flat surface $(X,|\omega|)$, with $Z(\omega)$ among its vertices. $X$ can then be presented as a quotient $\sim$ of the edges of a collection of triangles. One can check that the periods $\int_{w_i} \omega$ defining $\phi_{T,\omega}$ correspond to the vectors $\left\{v_i\right\} \in \mathbb{C}$, the edges of the polygon $P$. 

\end{subsection}

\begin{subsection}{$\mathrm{SL}(2,\mathbb{R})$-action and Transversal Directions}
\label{transversal_separation}

For a pair $(X,\omega)$, one can define the area of $(X,\omega)$

\begin{displaymath}
\mathrm{area}(X,\omega) =  \frac{i}{2} \int_X \omega \wedge \overline{\omega}.
\end{displaymath}
We denote by $\Omega_1\mathcal{M}_g(\sigma) \subset \Omega \mathcal{M}_g(\sigma)$ the unit-area locus. Assume $\left\{w_i\right\}_{i=1}^h$ forms a symplectic $\mathbb{Z}$-basis of $H_1(X,Z(\omega),\mathbb{Z})$. Choose a fundamental domain $\mathcal{F}$ for the action of $\mathrm{Mod}\left(S_g\right)$ on $\Omega_1\mathcal{T}_g(\sigma)$. The fundamental domain $\mathcal{F}$ is then identified with the quotient space $\Omega_1\mathcal{M}_{g}(\sigma)$. In period coordinates, there is a linear $\mathrm{SL}(2,\mathbb{R})$-action on $\mathcal{F}$; in particular, for a $2\times h$ real-valued matrix $B$ representing $\phi_{T,\omega}$ at a point $ x =(X,\omega) \in \Omega\mathcal{M}_{g}(\sigma)$, and $g \in \mathrm{SL}(2,\mathbb{R})$

\begin{displaymath}
g \cdot x = gBA(g,x)
\end{displaymath}
where the right-hand side is just matrix multiplication and $A: \mathrm{SL}(2,\mathbb{R}) \times \Omega_1\mathcal{M}_g(\sigma) \rightarrow \mathrm{Sp}(2g,\mathbb{Z}) \ltimes \mathbb{R}^h$ is the change of basis required to return $g\cdot x$ to the fundamental domain. The map $A$ is called the Kontsevich-Zorich cocycle. The Kontsevich-Zorich cocycle is not quite a cocycle in the dynamical sense due to the fact that the automorphism group $\mathrm{Aut}(X,\omega)$ is non-trivial (i.e., $\Omega\mathcal{M}_g(\sigma)$ is an orbifold). However, by considering for example, a level-3 structure \cite{eskin2018invariant}, $\Omega_1\mathcal{M}_g(\sigma)$ has a finite cover which is a manifold instead of an orbifold. Henceforth, we will assume we are working in the finite cover and continue to use the notation $\Omega_1\mathcal{M}_g(\sigma)$. In particular, for a pair $(X,\omega) \in \Omega_1\mathcal{M}_g(\sigma)$, the underlying Riemann surface $X$ will have no non-trivial automorphisms. Since applying a rotation matrix to $(X,\omega)$ doesn't change the complex structure on $X$, we may assume the Veech groups $\mathrm{SL}(X,\omega)$ contain no elliptic elements. This will be relevant in Subsection $\ref{finding_elements}$.\\

For any subset $\mathcal{N} \subset \Omega_1\mathcal{M}_g(\sigma)$, define

\begin{displaymath}
\mathbb{R}\mathcal{N} = \left\{(X,t\omega) \ \vert \ (X,\omega) \in \mathcal{N}, t \in \mathbb{R} \right\} \subset \mathbb{R}\Omega_1 \mathcal{M}_g(\sigma).
\end{displaymath}
Denote by $H^1$  the complex flat vector bundle over $\mathbb{R}\Omega_1\mathcal{M}_g(\sigma)$ whose fiber over any $(X,\omega)$ is \ $H^1(X,\mathbb{C})$, and by $H^1_{rel}$ the complex flat vector bundle whose fiber over $(X,\omega)$ is $H^1(X,Z(\omega),\mathbb{C})$. We will denote by $p:H^1_{rel} \mapsto H^1$ the forgetful map from relative to absolute cohomology.

\begin{definition}\label{definition2.1}
We call a subset $\tilde{\mathcal{N}} = \mathbb{R}\mathcal{N} \subset \mathbb{R}\Omega_1\mathcal{M}_g(\sigma)$ an affine invariant submanifold of \newline $\Omega_1\mathcal{M}_g(\sigma)$ if $\mathcal{N}$ is an analytic submanifold of $\Omega_1\mathcal{M}_g(\sigma)$, and locally in period coordinates, $\tilde{\mathcal{N}}$ is a complex-linear subspace of $\mathbb{C}^{2g + |Z(\omega)| -1}$ given by linear equations with real coefficients.
\end{definition}
The following observation is helpful for our purposes.

\begin{prop}\label{prop2.2}
The tangent bundle of $\tilde{\mathcal{N}}$ is determined in local period coordinates by a subspace 

\begin{displaymath}
T_{\mathbb{C}}(\tilde{\mathcal{N}}) = T_{\mathbb{R}}(\tilde{\mathcal{N}}) \otimes_{\mathbb{R}} \mathbb{C}
\end{displaymath}
where $T_{\mathbb{R}}(\tilde{\mathcal{N}}) \subset H^1(X,Z(\omega),\mathbb{R})$. Further, the tangent space to the stratum at $(X,\omega) \in \mathbb{R}\Omega_1\mathcal{M}_g(\sigma)$ is identified with $H^1(X,Z(\omega),\mathbb{C}) \cong H^1(X,Z(\omega),\mathbb{R}) \otimes_{\mathbb{R}} \mathbb{C}$ and the $\mathrm{SL}(2,\mathbb{R})$-action is on $\mathbb{C} \cong \mathbb{R}^2$ under this identification. In addition, if in local period coordinates $(X,\omega)$ is written as $ x = u + iv$, for $u$ and $v$ column vectors, then the $\mathbb{R}$-linear span of $u$ and $v$, $\mathcal{H}(x)$, is a subspace of $H^1(X,Z(\omega),\mathbb{R})$ so that the tangent space to the $\mathrm{SL}(2,\mathbb{R}$-orbit of $(X,\omega)$ is contained in $\mathcal{H}(x) \otimes \mathbb{C}$.
\end{prop}

\begin{definition}
\label{Hodge_inner_product}
The complex vector bundle $H^1$ has a symplectic intersection form given by 

\begin{displaymath}
\langle{\omega_1,\omega_2\rangle}:= \frac{i}{2}\int_X \omega_1 \wedge \overline{\omega_2}
\end{displaymath}
for $\omega_1, \omega_2$ in a fiber of $H^1.$ Henceforth, we will refer to a tangent vector $v \in H^1(X,Z(\omega),\mathbb{C})$ at a point $(X,\omega)$ as in a direction $\textit{tranversal}$ to the $\mathrm{SL}(2,\mathbb{R})$-orbit of $(X,\omega)$ if $p(v)$ is in the symplectic complement of the image under $p$ of the tangent space to the $\mathrm{SL}(2,\mathbb{R})$-orbit, $p(\mathbf{H}) \subset p(\mathcal{H}(x) \otimes_{\mathbb{R}} \mathbb{C})$. For any such $v \in H^1(X,Z(\omega),\mathbb{C})$ in this complement, we write $v \in p(\mathbf{H})^{\perp}$.

\end{definition}
By the preceeding discussion, we have a decomposition of the tangent bundle

\begin{equation}
\label{Equation2.3}
T_{\mathbb{C}}(\tilde{\mathcal{N}}) = T_{\mathbb{C}}^{st}(\tilde{N}) \oplus T_{\mathbb{C}}^{bal}(\tilde{\mathcal{N}})
\end{equation}
where $T_{\mathbb{C}}(\tilde{\mathcal{N}})^{bal}$ consists of the elements $v \in T_{\mathbb{C}}(\tilde{\mathcal{N}})$ so that $v \in p(\mathbf{H})^{\perp}$, using the notation in Definition $\ref{Hodge_inner_product}$. This decomposition is invariant under the action of $\mathrm{SL}(2,\mathbb{R})$, but is not invariant under parallel transport unless $\tilde{\mathcal{N}}$ is a Teichm{\"u}ller curve.

\begin{definition}[AGY-metric, \cite{AvilaExp}]
\label{defn_AGYnorm}
Let $x = (X,\omega) \in \Omega_1\mathcal{M}_g(\sigma)$. Identify the tangent space at $(X,\omega)$ with $H^1(X,Z(\omega),\mathbb{C})$. For a tangent vector $v \in H^1(X,Z(\omega),\mathbb{C})$, consider the norm

\begin{displaymath}
\lvert\lvert v \rvert\rvert_{x} :=\underset{\mathbf{s} \in \mathcal{S}} {\mathrm{sup}} \Big\lvert \frac{v(\mathbf{s})}{\mathrm{Hol}_{\omega}(\mathbf{s})}\Big\rvert
\end{displaymath}
where $\mathcal{S}$ denotes the set of saddle connections on $(X,\omega)$, and $\mathrm{Hol}_{\omega}(\mathbf{s}) \in \mathbb{C}$ is the holonomy vector associated to $\mathbf{s}$ for the form $\omega$. The set of holonomy vectors is always discrete. It is shown in \cite{AvilaExp} that this definition induces a complete Finsler metric on $\Omega_1\mathcal{M}_g(\sigma)$. Let $\kappa:[0,1] \rightarrow \Omega_1\mathcal{M}_g(\sigma)$ be a differentiable path and denote its AGY-length by

\begin{displaymath}
l(\kappa):= \int_0^1 \lvert \lvert \kappa'(t) \rvert \rvert_{\kappa(t)} \ dt.
\end{displaymath}
The distance function associated to this length will be denoted by $d_{AGY}$.

\end{definition}

\begin{lemma}[\cite{Avila2010SmallEO}]\label{lemma2.5}
 Let $\kappa:[0,1] \rightarrow \Omega_1\mathcal{M}_g(\sigma)$ be a differentiable path. Then for any tangent vector $v \in H^1(X,Z(\omega),\mathbb{C})$,
 \begin{displaymath}
 e^{-\mathrm{length}(\kappa)} \leq \frac{\lvert\lvert v \rvert\rvert_{\kappa(0)}}{\lvert\lvert v \rvert \rvert_{\kappa(1)}} \leq e^{\mathrm{length}(\kappa)}.
 \end{displaymath}
\end{lemma}

The following and its proof are essentially Proposition 5.3 in \cite{Avila2010SmallEO}; there, it is stated for vectors in the unstable space, but the proof goes through for any choice of tangent vectors. We briefly recapitulate the proof ingredients for the general case, as we will make extensive use of the AGY-norm.

\begin{lemma}["Exponential map", \cite{Avila2010SmallEO}]
\label{exponential_map}
Let $x = (X,\omega) \in \Omega_1\mathcal{M}_g(\sigma)$ and $v \in H^1(X,Z(\omega),\mathbb{C})$, a tangent vector. Consider the path $\kappa:[0,1] \rightarrow \Omega_1\mathcal{M}_g(\sigma)$ starting from $x$ satisfying $\kappa'(t) =v$ for all $t \in [0,1]$. Implicitly, we assume $\kappa(t)$ is well-defined for all $t \in [0,1]$. Define $\mathrm{exp}_{x}(v):= \kappa(1) \in \Omega_1\mathcal{M}_g(\sigma)$. Let $B(0,r)$ denote the ball of radius $r$ centered at the origin of $H^1(X,Z(\omega),\mathbb{C})$ with respect to the norm $\lvert \lvert \cdot \rvert\rvert_{x}$. Then, there exists an absolute constant $C_1 > 0$ and $C_2 = (C_1 +1)^2C_1$ such that the following hold:

\begin{enumerate}
\item for the rescaling $\omega \mapsto s\omega$, $s\in \mathbb{R}$, $l(\kappa) = l(s\kappa)$. That is, the AGY-metric is invariant under real scalings.
\item The exponential map $\mathrm{exp}_x:H^1(X,Z(\omega),\mathbb{C}) \rightarrow \Omega_1\mathcal{M}_g(\sigma)$ is well-defined over $B(0,1/C_1)$.
\item For every $v \in B(0,1/C_1)$,
\begin{displaymath}
d_{AGY}(x, \mathrm{exp}_{x}(v)) \leq C_1\lvert\lvert v \rvert\rvert_{x}.
\end{displaymath}
\item For every $v \in B(0,1/C_1)$ and every $w \in H^1(X,Z(\omega),\mathbb{C})$,
\begin{displaymath}
1/C_1 \leq \frac{\lvert\lvert w \rvert\rvert_{x}}{\lvert\lvert w \rvert\rvert_{\mathrm{exp}_{x}(v)}} \leq C_1.
\end{displaymath}
\item For $v \in B(0,1/C_2)$,
\begin{displaymath}
d_{AGY}(x, \mathrm{exp}_x(v)) \geq \lvert\lvert v \rvert\rvert_{\omega}/C_1.
\end{displaymath}
\end{enumerate}

\end{lemma}
\begin{proof}
For 1., we have

\begin{align*}
l(s\kappa) &= \int_0^1 \lvert \lvert s\kappa'(t) \rvert \rvert_{s\kappa(t)} \ dt \\
&= \int_0^1 \underset{\mathbf{s} \in \mathcal{S}} {\mathrm{sup}} \Bigg\lvert \frac{s\kappa'(t)(\mathbf{s})}{\mathrm{Hol}_{s\kappa(t)}(\mathbf{s})}\Bigg\rvert dt \\
&= \int_0^1 \underset{\mathbf{s} \in \mathcal{S}} {\mathrm{sup}} \Bigg\lvert \frac{s\kappa'(t)(\mathbf{s})}{s\mathrm{Hol}_{\kappa(t)}(\mathbf{s})}\Bigg\rvert dt \\
&= l(\kappa).
\end{align*}
For 2., let $\kappa:[0,1] \rightarrow \Omega_1\mathcal{M}_g(\sigma)$ be a path with $\kappa(0) = (X,\omega)$ and $\kappa'(t) = v \ \forall t \in [0,1]$. By Lemma $\ref{lemma2.5}$, we have

\begin{equation}\label{equation2.3}
\lvert \lvert \kappa'(t) \rvert\rvert_{\kappa(t)} = \lvert\lvert v \rvert\rvert_{x} e^{\int_0^t \lvert \lvert \kappa'(r) \rvert \rvert_{\kappa(r)} \ dr}.
\end{equation}
Consider the length function $L(t):\mathbb{R} \mapsto \mathbb{R}$ defined by $L(t) = \int_0^t \lvert \lvert \kappa'(r) \rvert \rvert_{\kappa(r)} \ dr$. By Equation $\ref{equation2.3}$, we have $L'(t) \leq e^{L(t)}\lvert\lvert v \rvert\rvert_{x}.$ Integrating both sides of this inequality and rearranging, we have

\begin{displaymath}
\lvert\lvert \kappa'(t) \rvert\rvert_{\kappa(t)} \leq \frac{\lvert\lvert v \rvert\rvert_{x}}{1-t\lvert\lvert v \rvert\rvert_{x}},
\end{displaymath}
which immediately implies that for any $C_1 > 1$, $\mathrm{exp}_\omega(v)$ is well-defined on the ball $B(0,1/C_1)$. Further, note that 

\begin{displaymath}
d_{AGY}(x,\mathrm{exp}_{x}(v)) \leq \int_0^1 \lvert\lvert \kappa'(t) \rvert\rvert_{\kappa(t)} \leq \frac{\lvert\lvert v \rvert\rvert_{x}}{(1-\lvert\lvert v \rvert\rvert_{x})},
\end{displaymath}
which implies $d_{AGY}(x,\mathrm{exp}_{x}(v)) \leq C_1 \lvert\lvert v \rvert\rvert_{x}$ for $v \in B(0,1/C_1)$, establishing 3. Since $\mathrm{exp}_{\omega}$ is well-defined over $B(0,1/C_1)$ by 2., another application of Lemma \ref{lemma2.5} gives

\begin{displaymath}
e^{-L(1)} \leq \frac{\lvert\lvert w \rvert\rvert_{x}}{\lvert \lvert w \rvert\rvert_{\mathrm{exp}_{x}(v)}} \leq e^{L(1)}
\end{displaymath}
for any $w \in H^1(X,Z(\omega),\mathbb{C})$, proving 4., since $L(1) \leq \mathrm{log}(C_1)$. Finally, we claim that for $C_2 = (C_1 +1)^2C_1 $, we have

\begin{displaymath}
d_{AGY}(x, \mathrm{exp}_{x}(v)) \geq \lvert\lvert v \rvert\rvert_{x}/C_1
\end{displaymath}
so long as $v \in B(0,1/C_2)$. To see this, consider a (nearly) length-minimzing path $\kappa$ connecting $(X,\omega)$ and $\mathrm{exp}_{\omega}(v)$. By part 3., we have

\begin{equation}\label{equation2.4}
l(\kappa) \leq C_1 \lvert\lvert v \rvert\rvert_{x} \leq \frac{1}{(C_1 + 1)^2}.
\end{equation}
Consider a lift of $\kappa$, $\tilde{\kappa}$ to the tangent space to the stratum identified with $H^1(X,Z(\omega),\mathbb{C})$. By part 4., we have $\lvert\lvert \tilde{\kappa}'(t) \rvert\rvert_{x} \leq C_1 \lvert\lvert \tilde{\kappa}'(t) \rvert\rvert_{\kappa(t)}$. This inequality can be integrated from $0$ to $t$ to obtain

\begin{equation}\label{equation2.5}
\lvert\lvert \tilde{\kappa}(t) \rvert\rvert_{x} \leq C_1l(\kappa) \leq C_1 \frac{1}{(C_1+1)^2} < 1/C_1.
\end{equation}
We conclude that $\tilde{\kappa}(t) \in B(0,1/C_1) \ \forall t \in [0,1]$ and in particular $\tilde{\kappa}(1)$ is well-defined and we have
\begin{align*}
\lvert\lvert \tilde{\kappa}(1)\rvert\rvert_{x} = \lvert\lvert v \rvert\rvert_{x} &\leq  C_1 l(\kappa)\\
&\leq C_1 d_{AGY}(x,\mathrm{exp}_{x}(v)),
\end{align*}
proving 5.
\end{proof}

\end{subsection}

\subsection{Unstable and stable foliations}
\label{unstable_stable_foliations}

In contrast to the homogeneous setting, the unstable and stable foliations for the geodesic flow are only locally defined. We review these definitions and their relationships to our objects of interest.

\begin{definition}[Proposition 4.1 \cite{Avila2010SmallEO}]
\label{Definition2.7}
Recall that for any $x = (X,\omega) \in \Omega\mathcal{M}_g(\sigma)$ there is an open neighborhood $U \subset \Omega\mathcal{M}_g(\sigma)$ containing it so that the period map $\phi_{T,\omega}: U \rightarrow H^1(X,Z(\omega),\mathbb{C}) \cong \mathbb{C}^{2g + Z(\omega) - 1}$ is an open analytic embedding. Henceforth, we suppress the notation for the triangulation $T$ and write $\phi_{\omega}$. Further, recall that the tangent space to $(X,\omega)$ in the stratum can be identified with $\newline$ $T_{\omega}\Omega\mathcal{M}_g(\sigma) \cong H^1(X,Z(\omega),\mathbb{C}) \cong H^1(X,Z(\omega),\mathbb{R}) \otimes_{\mathbb{R}} \mathbb{C} \cong H^1(X,Z(\omega),\mathbb{R}) \oplus H^1(X,Z(\omega),i\mathbb{R})$. Then, there is a decomposition of the tangent space

\begin{displaymath}
T_{x}\Omega\mathcal{M}_g(\sigma) \cong \mathbf{v}(x) \oplus E^{u}(x) \oplus E^{s}(x)
\end{displaymath}
where $\mathbf{v}(x)$ is the direction of the Teichm{\"u}ller geodesic flow, and

\begin{align*}
E^{u}(x) =& T_{x}\Omega\mathcal{M}_g(\sigma) \cap D\phi_{\omega}^{-1}(H^1(X,Z(\omega),\mathbb{R})), \\
E^{s}(x) =& T_{x}\Omega\mathcal{M}_g(\sigma) \cap D\phi_{\omega}^{-1}(H^1(X,Z(\omega),i\mathbb{R})).
\end{align*}
These subspaces are integrable, and depend smoothly on $(X,\omega)$. We will refer to the integral leaves of $E^{u}(x)$ and $E^{s}(x)$, respectively, as the unstable and stable manifolds $W^{u}(x)$ and $W^{s}(x)$. These are affine submanifolds of $\Omega\mathcal{M}_g(\sigma)$ in the sense of Definition $\ref{definition2.1}$,
\end{definition}

It follows from Lemma 3.1 in $\cite{smillie2023horospherical}$ that the foliations $W^{u}(x)$ and $W^{s}(x)$ are well-defined on the unit-area locus $\Omega_1\mathcal{M}_g(\sigma)$. 

\begin{definition}[Period boxes]
Let $(\tilde{X},\tilde{\omega}) \in \Omega_1\mathcal{T}_g(\sigma)$. Define, for every $r>0$,

\begin{displaymath}
R_r(\tilde{X},\tilde{\omega}):= \left\{\phi_{\omega}(\tilde{x},\tilde{\omega}) + u + iv \ : \ u,v \in H^1(X,Z(\omega),\mathbb{R} \ , \lvert\lvert u + iv \rvert\rvert \leq r \right\}.
\end{displaymath}
where we are again using the identification $H^1(X,Z(\omega),\mathbb{C}) \cong H^1(X,Z(\omega),\mathbb{R} \oplus H^1(X,Z(\omega),i\mathbb{R})$.
\end{definition}
Note that we may choose $r > 0$ small enough so that 

\begin{displaymath}
\phi^{-1}_{\omega}: R_r(\tilde{X},\tilde{\omega}) \cap \phi_{\omega}(\Omega_1\mathcal{T}_g(\sigma)) \mapsto \Omega_1\mathcal{T}_g(\sigma)
\end{displaymath}
is a homeomorphism. Set $B_r(\tilde{X},\tilde{\omega})= \phi^{-1}_{\omega}\left(R_r(\tilde{X},\tilde{\omega})\right)$. Observe that, by Lemma $\ref{exponential_map}$, $B_r(\tilde{X},\tilde{\omega})$ is well-defined for all $0< r< 1/C_1$ and all $(\tilde{X},\tilde{\omega}) \in \Omega_1\mathcal{T}_g(\sigma)$. Further, for any $(X,\omega) \in \Omega_1\mathcal{M}_g(\sigma)$, there exists $0 < r(\omega) < 1/C_1$ so that the restriction of the projection map from the Teichm{\"u}ller space to moduli space $\pi|_{B_{r(\omega)}(\tilde{X},\tilde{\omega}) } : \Omega_1\mathcal{T}_g(\sigma) \rightarrow \Omega_1\mathcal{M}_g(\sigma)$ is injective. For all $t, s \in \mathbb{R}$, let $a_t = \begin{bmatrix}
    e^{t/2} & 0 \\
    0 & e^{-t/2}
\end{bmatrix}$ and $u_s = \begin{bmatrix}
    1 & s \\
    0 & 1
\end{bmatrix}$. The action of the transpose of $u_s$ is given by $u_s^{\intercal} = \begin{bmatrix}
    1 & 0 \\
    s & 1
\end{bmatrix}$. Recall that the two one-parameter subgroups $a_t$ and $u_s$ of $\mathrm{SL}(2,\mathbb{R})$ define, respectively, the Teichm{\"u}ller geodesic flow, and horocycle flow on $\Omega_1\mathcal{M}_g(\sigma)$. The horocycle flow $u_s$ preserves the leaves of the unstable foliation $W^{u}(\omega)$ and the tranpose $u_s^{\intercal}$ preserves the leaves of the stable foliation $W^{s}(\omega)$. To see this explicitly, let $z = \begin{bmatrix}
x_1 & ... & x_n\\
y_1 & ... & y_n
\end{bmatrix}$, and $z' = \begin{bmatrix}
x_1' & ... & x_n'\\
y_1' & ... & y_n'
\end{bmatrix}$ be two points in the same local period coordinate chart. If we write $\phi_{\omega}(X,\omega) = u + iv$, the stable leaves of $W^{s}(x)$ are locally identified with $\phi_{\omega}(X,\omega) + iw$ and $w$ is some row $n$-vector, by Definition $\ref{Definition2.7}$. Therefore, we see that $z$ and $z'$ are in the same stable leaf if 

\begin{displaymath}
\begin{bmatrix}
x_1 & ... & x_n\\
y_1 & ... & y_n
\end{bmatrix} =  \begin{bmatrix}
x_1' & ... & x_n'\\
y_1' & ... & y_n'
\end{bmatrix} + \begin{bmatrix}
0 & ... & 0\\
w_1 & ... & w_n
\end{bmatrix}.
\end{displaymath}
The analogous observation is true for the unstable leaves. Let $e^{-0.01t} < \beta < 1$.  For every $t \geq 0$, define

\begin{displaymath}
E_{t} = B_{\beta}^s \cdot a_{t} \cdot \left\{ u_r: r \in [0,1] \right\} \subset \mathrm{SL}(2,\mathbb{R})
\end{displaymath}
where $B_{\beta}:= \left\{u_s^{\intercal}: |s| \leq \beta \right\} \cdot \left\{a_t: |t| \leq \beta \right\}$. In words, this is a "smearing" by a small amount of $\mathrm{SL}(2,\mathbb{R})$ of a geodesic push of length $t$ of a length-1 horocycle segment. As we have just observed, this is in fact akin to a Margulis thickening in the unstable direction, but only along the $\mathrm{SL}(2,\mathbb{R})$-orbit, rather than the ambient space. We recall that for any such $E_t$ as above, and any $(X,\omega) \in \Omega_1\mathcal{M}_g(\sigma)$, we define a map $E_t \rightarrow \Omega_1\mathcal{M}_g(\sigma)$ given by $h \mapsto h\cdot (X,\omega)$, since any $h \in E_t$ is in particular an element of \ $\mathrm{SL}(2,\mathbb{R})$. We see that this map is injective if, for any lift of $(X,\omega)$, $(\tilde{X},\tilde{\omega})$, the restricted projection map $\pi|_{B_r(\tilde{X},\tilde{\omega}) \cap E_t \cdot (\tilde{X},\tilde{\omega}) } : \Omega_1\mathcal{T}_g(\sigma) \rightarrow \Omega_1\mathcal{M}_g(\sigma)$ is injective.

\section{Quantitative non-divergence}
\label{Section3}

We collect quantitative non-divergence results in a form convenient for our application.

\begin{definition}[Injectivity radius]
\label{inj_radius_def}
For $(X,\omega) \in \Omega_1\mathcal{T}_2(2)$, let $\mathcal{S}_{\omega}$ denote the set of saddle connections of $\omega$ on $X$. We will use $l_{\omega,\mathbf{s}}=\lvert \mathrm{Hol}_{\omega}(\mathbf{s})\rvert$ to denote the length of $\mathbf{s}$ in the conical flat metric $|\omega|$. Recall that we have a quotient map $\pi:\Omega_1\mathcal{T}_2(2) \rightarrow \Omega_1\mathcal{M}_2(2)$. We define the injectivity radius of $(X,\omega) \in \Omega_1\mathcal{T}_2(2)$, denoted by $\mathrm{inj}(X,\omega)$ to be the infinum 

\begin{displaymath}
\mathrm{inj}(X,\omega) = \underset{s \in \mathcal{S}_{\omega}}{\mathrm{inf}}\lvert \mathrm{Hol}_{\omega}(\mathbf{s}) \rvert.
\end{displaymath}
Define for $\epsilon > 0$,

\begin{displaymath}
\Omega_1\mathcal{M}_{\epsilon} = \pi \left(\left\{(X,\omega) \in \Omega_1\mathcal{T}_2(2) \ | \ \mathrm{inj}(X,\omega) \geq \epsilon \right\}\right).
\end{displaymath}
If $(X,\omega) \in \Omega_1\mathcal{M}_{\epsilon}$ (resp. $(X,\omega) \notin \Omega_1\mathcal{M}_{\epsilon})$, we will also say $\mathrm{inj}(X,\omega) \geq \epsilon$ (resp. $\mathrm{inj}(X,\omega) < \epsilon)$.
\end{definition}

\begin{theorem}[Minsky-Weiss \cite{minsky2002nondivergence}]\label{theorem4.2}
There are positive constants $\kappa_1$, $\kappa_2$, $\alpha$ ($\alpha \leq 1)$ depending only on the stratum $\Omega_1\mathcal{M}_g(\sigma)$ such that if every $(X,\omega) \in \Omega_1\mathcal{M}_g(\sigma)$, $s \in \mathcal{S}_{\omega}$, an interval $I \subset \mathbb{R}$, and \ $\rho' > 0$ satisfy:

\begin{displaymath}
\underset{s \in I}{\mathrm{sup}} \ l_{u_s\omega,\mathbf{s}} \geq \rho',
\end{displaymath}
then for any $0 < \epsilon < \kappa_1\rho'$ we have
\begin{displaymath}
{\big|\left\{r \in I : \mathrm{inj}(u_s (X,\omega)) < \epsilon \right\}\big|} < \kappa_2 \cdot \left(\frac{\epsilon}{\rho'}\right)^{\alpha}|I|.
\end{displaymath}
\end{theorem}
For clarity, we note that the condition $\underset{s \in I}{\mathrm{sup}} \ l_{u_s\omega,\mathbf{s}} \geq \rho'$ just means that there is no saddle connection on the surface $(X,\omega)$ whose length in the conical flat metric remains below $\rho'$ over a time-length $|I|$ flow from $(X,\omega)$. A consequence of the arguments in this section is a proof of the following lemma.

\begin{lemma}
\label{Main_quant_divergence}
There exist constants $\alpha, C_3, C_4 >0$ (depending only on the stratum) so that the following property holds. Let $0 < \epsilon, \eta < 1$ and $(X,\omega) \in \Omega_1\mathcal{M}_g(\sigma).$ Let $I \subset [-10,10]$ be an interval and assume $|I|\geq \eta$. Then, 

\begin{displaymath}
{\big|\left\{s \in I : \mathrm{inj}\left(a_t u_s (X,\omega)\right) < {\epsilon}\right\}\big|} < C_4  {\epsilon}^{\alpha}|I|,
\end{displaymath}
so long as $t \geq \big\lvert \mathrm{log}\left(\eta  \ \mathrm{inj}(X,\omega)\right)^{-1}\big\rvert + C_3$. The constants are explicit; $\alpha$ is the $\alpha$ appearing in Theorem $\ref{theorem4.2}$, $C_3 = \mathrm{log}\left(\frac{1}{\kappa_1}\right) + \mathrm{log} (\sqrt{2})$, and $C_4 = \frac{\kappa_2}{\kappa_1^{\alpha}}$ where $\kappa_1,\kappa_2$ are the constants appearing in Theorem $\ref{theorem4.2}$.
\end{lemma}

Although we only require Lemma $\ref{Main_quant_divergence}$ for subsequent use in this paper, we record an auxiliary statement, to emphasize the analogy to quantitative non-divergence results of this kind in the homogeneous setting. We refer the reader to Section 3 of \cite{lindenstrauss2022polynomial1}, where the corresponding statements are proven for the cases of $\mathrm{SL}(2,\mathbb{R}) \times \mathrm{SL}(2,\mathbb{R})$ and $\mathrm{SL}(2,\mathbb{C})$. We recall the following definitions due to Minsky-Weiss.

\begin{definition}[Good functions, Definition 3.1 \cite{minsky2002nondivergence}]
\label{Definition4.4}
Let $\mathcal{F}$ be a collection of continuous functions $\mathbb{R} \rightarrow \mathbb{R}_{+}$, and $I \subset \mathbb{R}$ some interval. For $\theta > 0$ and $f \in \mathcal{F}$, make the following definitions:
\begin{align*}
I_{f,\theta} &= \left\{s \in I: f(s) < \theta \right\} \\
I_{\mathcal{F},\theta} &= \left\{s \in I : \exists f \in \mathcal{F},\ f(s) < \theta \right\} \\
\lvert \lvert f \rvert\rvert_{I} &= \underset{s \in I}{\mathrm{sup}} \ f(s).
\end{align*}
Let $\kappa, \alpha, \rho > 0$. We designate $\mathcal{F}$ as $(\kappa,\alpha,\rho)$-good if it satisfies the following property. For any interval $I \subset \mathbb{R}$ and any $f \in \mathcal{F}$, we have, for $0 < \epsilon < \rho$,
\begin{displaymath}
\frac{\big\lvert I_{f,\epsilon}\big\rvert}{\lvert I \rvert} \leq \kappa \left(\frac{\epsilon}{\lvert\lvert f \rvert\rvert_{I}}\right)^{\alpha}.
\end{displaymath}
We call $\mathcal{F}$ $(\kappa,\alpha)$-good if it is $(\kappa,\alpha,\rho)$-good for all $\rho$.
\end{definition}

\begin{definition}[Sparse-covering property]
\label{Definition4.5}
We say a pair consisting of a collection of continuous functions $\mathcal{F}$ and an interval $I$ as above, $(\mathcal{F},I)$, satisfies the sparse-covering property if the following conditions hold. There exist $\kappa, \alpha, \rho, M > 0$ such that

\begin{enumerate}
\item $\mathcal{F}$ is $(\kappa,\alpha,\rho)$-good.
\item For every $f \in \mathcal{F}$, $\lvert\lvert f \rvert\rvert_{I} \geq \rho$.
\item For every $s \in I$,
\begin{displaymath}
\# \left\{f \in \mathcal{F}: f(s) < \rho \right\} \leq M.
\end{displaymath}
\end{enumerate}
\end{definition}

\begin{prop}
Let $(F,I)$ satisfy the sparse-covering property in the sense of Definition \ref{Definition4.5} for the constants $\kappa, \alpha, \rho, M > 0$. Then, for every $0 < \epsilon < \rho$,

\begin{displaymath}
\frac{\big\lvert I_{\mathcal{F},\epsilon}\big\rvert}{\lvert I \rvert} \leq \kappa M \left(\frac{\epsilon}{\lvert\lvert f \rvert\rvert_{I}}\right)^{\alpha}.
\end{displaymath}
\end{prop}

\begin{prop}[\cite{minsky2002nondivergence}]
For any $(X,\omega) \in \Omega_1\mathcal{T}_g(\sigma)$, and $\mathbf{s} \in \mathcal{S}_{\omega}$, define a function $l_{\omega,\mathbf{s}}(s): \mathbb{R} \rightarrow \mathbb{R}_{+}$ by $l_{\omega,\mathbf{s}}(s) =  l_{u_s \omega,\mathbf{s}}$. That is, for a point in moduli space, and a saddle connection on the surface, consider the function measuring the length of the saddle connection along a horocycle segment. Consider the collection of functions 

\begin{displaymath}
\mathcal{F} = \left\{l_{\omega,\mathbf{s}} : (X,\omega) \in \Omega_1\mathcal{T}_g(\sigma), \mathbf{s} \in \mathcal{S}_{\omega}\right\}.
\end{displaymath}
$\mathcal{F}$ is $(2,1)$-good in the sense of Definition $\ref{Definition4.4}$.
\end{prop}
Unfortunately, there are $I \subset \mathbb{R}$ for which $(\mathcal{F},I)$ doesn't satisfy the sparse-covering property. However, the arguments in Section 6 of Minsky-Weiss \cite{minsky2002nondivergence} show:

\begin{lemma}
\label{Lemma4.8}
There exists a subset $\mathcal{F}_0 \subset \mathcal{F} = \left\{l_{\omega,\mathbf{s}} : (X,\omega) \in \Omega_1\mathcal{T}_g(\sigma), \mathbf{s} \in \mathcal{S}_{\omega}\right\}. $ so that $(\mathcal{F}_0,I)$ satisfies the sparse-covering property for any $I \subset \mathbb{R}$ and for which the following property holds. There are positive constants $\kappa_1$, $\kappa_2$, $\alpha$ depending only on the stratum $\Omega_1\mathcal{M}_g(\sigma)$ such that if every $f(s) \in \mathcal{F}_0$, an interval $I \subset \mathbb{R}$, and \ $\rho' > 0$ satisfy:

\begin{displaymath}
\underset{s \in I}{\mathrm{sup}} \ f(s) \geq \rho',
\end{displaymath}
then for any $0 < \epsilon < \kappa_1\rho'$ we have
\begin{displaymath}
{\big|\left\{s \in I : \mathrm{inj}(u_s (X,\omega)) < \epsilon \right\}\big|} < \kappa_2 \cdot \left(\frac{\epsilon}{\rho'}\right)^{\alpha}|I|.
\end{displaymath}
\end{lemma}

\begin{lemma}
\label{Lemma4.9}
There exists a subset $\mathcal{F}_0 \subset \mathcal{F} = \left\{l_{\omega,\mathbf{s}} : (X,\omega) \in \Omega_1\mathcal{T}_g(\sigma), \mathbf{s} \in \mathcal{S}_{\omega}\right\}. $ so that $(\mathcal{F}_0,I)$ satisfies the sparse-covering property for any $I \subset \mathbb{R}$ and for which the following property holds. There are positive constants $\kappa_1$, $\kappa_2$, $\alpha$ depending only on the stratum $\Omega_1\mathcal{M}_g(\sigma)$ such that if every $f(s) \in \mathcal{F}_0$, a symmetric interval $I \subset \mathbb{R}$ with $\lvert I \rvert > \eta$, and \ $\rho' > 0$ satisfy:

\begin{displaymath}
\underset{f \in \mathcal{F}_0}{\mathrm{inf}} \ f(0) \geq \rho',
\end{displaymath}
then for any $0 < \epsilon, \eta < 1$ we have
\begin{displaymath}
{\big|\left\{r \in I : \mathrm{inj}(a_t u_s (X,\omega)) < \epsilon \right\}\big|} < \kappa_2 \cdot \left({\epsilon}\right)^{\alpha}|I|,
\end{displaymath}
provided $t \geq \mathrm{max} \left\{ \Big\lvert \mathrm{log}\left(\sqrt{\frac{\rho'}{2}}\right)^{-1} \Big\rvert,   \Big\lvert \mathrm{log} \ \left(\eta \sqrt{\frac{\rho'}{2}}\right)^{-1} \Big\rvert \right\}$.
\end{lemma}

\begin{proof}
We first note that since $\underset{f \in \mathcal{F}_0}{\mathrm{inf}} \ f(0) \geq \rho'$, we have $\underset{f \in \mathcal{F}_0}{\mathrm{sup}} \ f(0) \geq \rho'$. That is, the longest saddle connection in the family $\mathcal{F}_0$ on $(X,\omega)$ is of length greater than $\rho'$. Let $\lvert \mathrm{Re}_{\omega}(\mathbf{s}) \rvert$, $\lvert \mathrm{Im}_{\omega}(\mathbf{s}) \rvert$, denote, respectively the horizontal and vertical lengths of $s$ with respect to $\omega$. Then,

\begin{align}
\label{equation4.1}
\underset{r \in I}{\mathrm{sup}} \ \lvert \mathrm{Re}_{a_tu_r\omega}(\mathbf{s}) \rvert &\geq e^{t} \lvert \mathrm{Re}_{\omega}(\mathbf{s}) \rvert.
\end{align}
Since $\lvert \mathrm{Hol}_{\omega}(\mathbf{s}) \rvert \geq \rho'$, $\mathrm{max} \left\{\lvert \mathrm{Re}_{\omega}(\mathbf{s})\rvert, \lvert \mathrm{Im}_{\omega}(\mathbf{s})\rvert \right\} \geq \sqrt{\frac{\rho'}{2}}$. Suppose $\lvert \mathrm{Re}_{\omega}(\mathbf{s}) \rvert \geq  \sqrt{\frac{\rho'}{2}}$. Then, we have,

\begin{align}
\underset{s \in I}{\mathrm{sup}} \ l_{a_tu_s\omega,\mathbf{s}} \geq
\underset{s \in I}{\mathrm{sup}} \ \lvert \mathrm{Re}_{a_tu_s\omega}(\mathbf{s}) \rvert &\geq e^{t} \lvert \mathrm{Re}_{\omega}(\mathbf{s}) \rvert \\
&\geq e^{t} \sqrt{\frac{\rho'}{2}}.
\end{align}
Hence, $\underset{s \in I}{\mathrm{sup}} \ l_{a_tu_s\omega,\mathbf{s}} \geq \frac{1}{\kappa_1}$, so long as  $t \geq \Big\lvert \left(\mathrm{log}\sqrt{\frac{\rho'}{2}}\right)^{-1} \Big\rvert + \big\lvert \mathrm{log}\left(\frac{1}{\kappa_1}\right)\big\rvert$. On the other hand, suppose $\lvert \mathrm{Im}_{\omega}(\mathbf{s}) \rvert \geq  \sqrt{\frac{\rho'}{2}}$. Then, 

\begin{align}
\underset{s \in I}{\mathrm{sup}} \ \lvert \mathrm{Im}_{a_tu_s\omega}(\mathbf{s}) \rvert &\geq e^{t} \eta \lvert \mathrm{Im}_{\omega}(\mathbf{s}) \rvert.
\end{align}
Therefore, $\underset{s \in I}{\mathrm{sup}} \ l_{a_tu_s\omega,\mathbf{s}} \geq \frac{1}{\kappa_1}$, so far as  $t \geq \Big\lvert \left(\mathrm{log} \ \eta \sqrt{\frac{\rho'}{2}}\right)^{-1} \Big\rvert + \big\lvert \mathrm{log}\left(\frac{1}{\kappa_1}\right)\big\rvert$. By Lemma $\ref{Lemma4.8}$, and the fact that $a_tu_s = u_{se^{2t}}a_t$, we have

\begin{displaymath}
\frac{{\big|\left\{s \in I : \mathrm{inj}(a_t u_s (X,\omega)) < \epsilon \right\}\big|}}{\lvert I \rvert} = \frac{{\big|\left\{r \in I' : \mathrm{inj}(u_sa_t (X,\omega)) < \epsilon \right\}\big|}}{\lvert I' \rvert} < \frac{\kappa_2}{\kappa_1^{\alpha}} \cdot \left({\epsilon}\right)^{\alpha},
\end{displaymath}
where $I':= e^{2t}I$, given that $t \geq \Big\lvert \left(\mathrm{log}\sqrt{\frac{\rho'}{2}}\right)^{-1} \Big\rvert + \big\lvert \mathrm{log}\left(\frac{1}{\kappa_1}\right)\big\rvert$ (recall $\eta$ < 1).
\end{proof}

\begin{proof}[Proof of Lemma \ref{Main_quant_divergence}]
Note that, by the definition of injectivity radius, if $\mathrm{inj}(X,\omega) \geq \rho'$, then using the notation of Lemma $\ref{Lemma4.9}$, $\underset{f \in \mathcal{F}_0}{\mathrm{inf}} \ f(0) \geq \rho'$. Lemma \ref{Main_quant_divergence} follows by the expressions given for the explicit constants $C_3$ and $C_4$ in the statement of the Lemma.
\end{proof}

\section{Margulis function}
\label{Section4}

\subsection{Non-divergence and injectivity estimates}

We combine some results in the literature into a proposition quantifying the injectivity of the projection map from Teichm{\"u}ller space to moduli space as a point varies along an orbit.

\begin{theorem}
[\cite{Eskin2001AsymptoticFO},\cite{athreya2006quantitative}]
\label{athreya_function}
There exists a continuous, proper, $\mathrm{SO}(2)$-invariant function $u: \Omega_1\mathcal{M}_g(\sigma) \rightarrow [2,\infty)$ such that there exists a constant $\kappa_3$ depending only on the stratum so that for all $x = (X,\omega) \in \Omega_1\mathcal{M}_g(\sigma)$ and all $t > 0$,
\begin{equation}
\label{athreya_inquality}
e^{-\kappa_3 t}u(x) \leq u(a_tx) \leq e^{\kappa_3 t}u(x).
\end{equation}

\end{theorem}
Recall that for $\tilde{x} = (\tilde{X},\tilde{\omega})$ we defined in Subsection \ref{unstable_stable_foliations} the restricted projection map $\pi|_{B_{r}(\tilde{X},\tilde{\omega}) } : \Omega_1\mathcal{T}_g(\sigma) \rightarrow \Omega_1\mathcal{M}_g(\sigma)$. By non-divergence statements close to those in Section \ref{Section3}, the following is proven in \cite{EMM_effective_simple}.

\begin{lemma}[Lemma 2.6, \cite{EMM_effective_simple}]
There exists a constant $C_5 > 0$ depending only on the stratum so that for all $\tilde{x} = (\tilde{X},\tilde{\omega}) \in \Omega_1\mathcal{T}_g(\sigma)$ and every $0 < r \leq u(x)^{-C_5}$ the restricted projection map $\pi|_{B_{r}(\tilde{x}) }$ is injective.
\end{lemma}

As a consequence, we have the following proposition.

\begin{prop}
\label{inj_radius_proposition}
    For any $t>0$ and any $\tilde{x} = (\tilde{X},\tilde{\omega}) \in \Omega_1\mathcal{T}_g(\sigma)$, the restricted projection map $\pi |_{B_{r}({a_t \tilde{x}})}$ is injective for $r < e^{-C_5\kappa_3 t}u(x)$.
\end{prop}

\subsection{Non-uniformly hyperbolic dynamics of the geodesic flow}
\label{non-uniform-dynamics}

Recall the Kontsevich-Zorich cocyle defined in Subsection $\ref{transversal_separation}$, $A: \mathrm{SL}(2,\mathbb{R}) \times \Omega_1\mathcal{M}_g(\sigma) \rightarrow \mathrm{Sp}(2g,\mathbb{Z}) \ltimes \mathbb{R}^h$. For $x = (X,\omega) \in \Omega_1\mathcal{M}_g(\sigma)$ and $v \in H^1(X,\mathbb{R}) $, its Lyapunov exponents $\lambda_i$ are  

\begin{displaymath}
\lambda_i = \lim\limits_{t \rightarrow \infty}\frac{1}{t}\mathrm{log}\frac{||A(a_t,x)v||}{||v||}.
\end{displaymath}
For $\mu_{MV}$ the Masur-Veech measure on a connected component $\mathcal{C}$ of a stratum $\Omega_1\mathcal{M}_g(\sigma)$, a foundational result of Forni \cite{forni2002deviation} gives 

\begin{theorem}
\label{Forni_spectral_gap}
The Lyapunov exponents of the Kontsevich-Zorich cocycle for $\mu_{MV}$ are all non-zero and distinct.
\end{theorem}
Theorem $\ref{Forni_spectral_gap}$ implies non-uniform hyperbolicity for the geodesic flow for the Hodge norm, which we will now define. Recall the inner product from Definition \ref{Hodge_inner_product} given by $\langle{\omega_1,\omega_2\rangle}:= \frac{i}{2}\int_X \omega_1 \wedge \overline{\omega_2}$, and the forgetful map $p:H^1(X,Z(\omega),\mathbb{R}) \rightarrow H^1(X,\mathbb{R})$ from relative to absolute homology. Denote by $||\cdot||_H$ the norm induced by this inner product on $H^1(X,\mathbb{R})$ and use the same notation to define a norm on $H^1(X,Z(\omega),\mathbb{R})$ via

\begin{displaymath}
||v||_H = ||p(v)||_H + \sum\limits_{(p,p') \in Z(\omega) \times Z(\omega)} \bigg\lvert \int_{\gamma_{p,p'}} (v - \mathbf{h}) \bigg\rvert
\end{displaymath}
where $\mathbf{h}$ is the harmonic 1-form representing $p(v)$ and $\gamma_{p,p'}$ is any path connecting $p$ to $p'$. The norm is independent of the choice of path. We will denote the Hodge norm distance by $d_{H}(\cdot,\cdot)$.

\begin{corollary}
\label{nonuniform_hyperbolicity}
For $\mu_{MV}$-a.e. $x = (X,\omega) \in \Omega_1\mathcal{M}_g(\sigma)$ and the decomposition of the tangent bundle $T_{x}\Omega\mathcal{M}_g(\sigma) \cong \mathbf{v}(x) \oplus E^{u}(x) \oplus E^{s}(x)$, we have the following. There is an $a_t$-invariant function $\lambda:\Omega_1\mathcal{M}_g(\sigma) \rightarrow \mathbb{R}^{+}$ with $\lambda(x) < 1$, so that for every $e_0 > 0$, there is an $a_t$-invariant function $E:\Omega_1\mathcal{M}_g(\sigma) \rightarrow (0,\epsilon_0) $ and a function $C(x)>0$ on $\Omega_1\mathcal{M}_g(\sigma)$ where

\begin{displaymath}
\label{non-uniform_condition}
||D_{a_t}(x)v||_{H} \leq C(x)(\lambda(x))^{t}||v||_{H}
\end{displaymath}
for any $v \in E^{s}(x)$, and $D$ is the derivative. Further, $C(a_t x) \leq C(x)e^{E(x)t}$.

\end{corollary}

\subsection{Norm comparisons and closing lemma}
\label{Norm_comparisons_closing}

It will be helpful for our purposes to know that the Teichm{\"u}ller norm $||\cdot||_{\mathcal{T}}$ defined in subsection $\ref{Section2.1}$ is commensurable with the AGY-norm $||\cdot||$ defined in Definition $\ref{defn_AGYnorm}$ and the Hodge norm defined in Section \ref{non-uniform-dynamics}, on compact sets. We would like a quantitative comparison of these norms, with explicit dependence on the compact set. For this, we will use the main results in \cite{SU2024108839} due to Su and Zhang and the main results in \cite{Hodge_Teich} due to Kahn and Wright.\\

 For a quadratic differential $q$ on a Riemann surface $X$, there exists a canonical double cover $\hat{X}$ and an Abelian differential $\omega$ on $\hat{X}$ so that the pullback of $q$ to $\hat{X}$ is $\omega^2$. The set of regular points includes the even zeroes of $q$, but the cover has ramification at the odd zeroes and poles. An even zero of $q$ of order $n_i$ corresponds to two zeroes of $\omega$ of orders $n_i /2$, and an odd zero of $q$ of order $n_i$ corresponds to a zero of order $n_i + 1$. The Deck group of $\hat{X}$ is an involution $\tau$; a conformal automorphism of $\hat{X}$ such that $\tau^{*}(\omega) = -\omega$. For $g \geq 2$, we will denote by $\mathcal{Q}_1\mathcal{M}_g(1,...,1)$ the principal stratum of unit-area quadratic differentials; that is, the space of quadratic differentials $(X,q)$ with $4g-4$ simple zeroes. The tangent space of $\mathcal{Q}_1\mathcal{M}_g(1,...,1)$ at $(X,q)$ can be identified with $H^1_{-1}(\hat{X},Z(q),\mathbb{C})$, the $-1$ eigenspace for the involution action on $H^1(\hat{X},Z(q),\mathbb{C})$, where $Z(q)$ denotes the zeros of $q$. When $q$ has only odd zeroes, $H^1_{-1}(\hat{X},Z(q),\mathbb{C})$ can be identified with $H^1_{-1}(\hat{X},\mathbb{C})$. Observe that in this context, $\mathrm{inj}(X,q)$ is still well-defined, as is the norm in Definition \ref{defn_AGYnorm}. Let $\psi:\mathcal{Q}_1\mathcal{M}_g(1,...,1) \rightarrow \mathcal{M}_g$ be the natural map $(X,q) \mapsto X$, and denote by $D$ its derivative.

 \begin{theorem}[Kahn-Wright, \cite{Hodge_Teich}]
 \label{Hodge_Teich_norm_comparison}
 Let $(X,q) \in \mathcal{Q}_1\mathcal{M}_g(1,...,1)$, and $P:\hat{X}\rightarrow X$ such that $P^{*}q = \omega^2$. Then, for any $v \in H^1_{-1}(\hat{X},\mathbb{C})$, we have
\begin{equation}
 \label{norm_comparison_equation_H&T}
 ||v||_H \leq ||D\psi(v)||_{\mathcal{T}} \leq \frac{4}{\sqrt{\mathrm{inj}(X,q)}}||v||_H.
 \end{equation}

 \end{theorem}

 \begin{theorem}[Su-Zhang,\cite{SU2024108839}]
 \label{theorem_norm_comparison}
 Let $(X,q) \in \mathcal{Q}_1\mathcal{M}_g(1,...,1)$, and $P:\hat{X}\rightarrow X$ such that $P^{*}q = \omega^2$. Then, for any $v \in H^1_{-1}(\hat{X},\mathbb{C})$, we have

 \begin{equation}
 \label{norm_comparison_equation}
 \frac{\mathrm{inj}({X},q)}{4\sqrt{2}}||v|| \leq ||D\psi(v)||_{\mathcal{T}} \leq \frac{16}{\sqrt{\mathrm{inj}(X,q)\pi}}||v||.
 \end{equation}
 \end{theorem}
 
We remark that every $(X,\omega) \in \Omega_1\mathcal{M}_2(2)$ is the holonomy double cover of a unique surface in the principal stratum (with possibly simple poles) $\mathcal{Q}_1\mathcal{M}_g(1,-1,-1,-1,-1,-1)$, and Theorems $\ref{Hodge_Teich_norm_comparison}$ and $\ref{theorem_norm_comparison}$ apply in this case. See Subsection 3.3 in \cite{Hodge_Teich} for an application of this fact to the sharpness of the bounds in Theorem \ref{Hodge_Teich_norm_comparison}. Denote the Teichm{\"u}ller distance by $d_{\mathcal{T}}(\cdot,\cdot)$. With this notion of distance, we require a closing lemma for the hyperbolic action of $a_t$ on a stratum. This statement weakly generalizes the usual Anosov closing lemma for the geodesic flow $a_t$ on the unit tangent bundle of a hyperbolic surface $T^{1}M = \mathrm{PSL}(2,\mathbb{R})/\Gamma$. Roughly, the Anosov closing lemma in this context states that if for $x \in T^{1}M$ and $t>0$, $d(x,a_t(x))< \epsilon$, then there exists an $x_0 \in T^{1}M$ lying on a periodic orbit, such that $d(x,x_0)\ll \epsilon$. Since the geodesic flow on a stratum isn't uniformly hyperbolic, but only uniformly hyperbolic in compact sets in the sense of Corollary \ref{nonuniform_hyperbolicity}, this requires a slight modification. This following is a consequence of the arguments due to Eskin-Mirzakhani-Rafi \cite{eskin2019counting}, Lemma 8.1. and to Hamenst{\"a}dt \cite{Bowen_T}

\begin{theorem}[\cite{eskin2019counting},Lemma 8.1.,\cite{Bowen_T}]
\label{Anosov_closing_lemma}
Let $\mathcal{K} \subset \Omega_1\mathcal{M}_g(\sigma)$ be a compact set. Given $(X,\omega) \in \mathcal{K}$ and $\delta > 0$, there exists $t_0 > 0$ so that the following holds. Suppose that $a_t(X,\omega)$ stays in $\mathcal{K}$ for all $t \geq 0$ and $d_{H}((X,\omega),a_t(X,\omega))< \delta$ for $t > t_0$. Then, there exists nearby $(X',\omega')$ lying on a closed Teichm{\"u}ller geodesic of length within \ $\delta$ of $t$, and $d_{\mathcal{T}}((X,\omega),(X',\omega')) < \delta$.
\end{theorem}

 The proof relies on following the argument of the usual Anosov closing lemma for uniformly hyperbolic flows, which itself relies on the contraction mapping principle. Note that by Corollary \ref{nonuniform_hyperbolicity}, if one chooses $x \in \mathcal{K} = \Omega_1\mathcal{M}_{\epsilon}$, $\epsilon_0$ small enough,  and $\lambda(x) = \lambda(a_t x) = e^{\kappa t}$, we may set $\epsilon_0 + \kappa = \epsilon_x + \kappa'$ where $\epsilon_x$ and $\kappa'$ are both negative. Using the fact that $C(a_t x) \leq C(x)e^{\epsilon_0 t}$ we have 

\begin{displaymath}
||D_{a_t}(x)v||_H \leq e^{\kappa't}||v||_H
\end{displaymath}
if $C(x)e^{\epsilon_x t} < 1$, i.e., $t > \mathrm{log}(1/C(x))\frac{1}{\epsilon_x}$, which is hyperbolicity with hyperbolic factor $e^{\kappa't}$. Observe that $C(x)$ and the size of an $\epsilon_x$ one can take depend only on the injectivity radius of $x$. We also remark that Theorem \ref{Anosov_closing_lemma} is stated for the geodesic flow $a_t$, but is valid for any one-parameter hyperbolic action. The following theorem of McMullen is crucial to our method, and is true only in stratum $\Omega_1\mathcal{M}_2(2)$. 

\begin{theorem}[\cite{McMullen_hyperbolic}, Theorem 5.8]
\label{Veech_hyperbolic}
Let $(X,\omega) \in \Omega_1\mathcal{M}_2(2)$ and suppose $\gamma \in \mathrm{SL}(X,\omega)$ is a hyperbolic element of \ $\mathrm{SL}(2,\mathbb{R})$. Then, $(X,\omega)$ lies on a Teichm{\"u}ller curve.
\end{theorem}

\subsection{Finding elements of the Veech group}
\label{finding_elements}

Henceforth, for $v \in H^1(X,Z(\omega),\mathbb{C})$, we will suppress the reference to the tangent space point, and simply write $||v||$. Recall that we have a decomposition

\begin{displaymath}
T_{\mathbb{C}}(\Omega_1\mathcal{M}_g(2)) = T_{\mathbb{C}}^{st}(\Omega_1\mathcal{M}_g(2)) \oplus T_{\mathbb{C}}^{bal}(\Omega_1\mathcal{M}_g(2)).
\end{displaymath}
 We also remind the reader that for $e^{-0.01t} < \beta < 1$ and every $t \geq 0$, we consider the "smeared" geodesic push

\begin{displaymath}
E_{t} = B_{\beta} \cdot a_{t} \cdot \left\{ u_s: s \in [0,1] \right\} \subset \mathrm{SL}(2,\mathbb{R})
\end{displaymath}
where $B_{\beta}:= \left\{u_s^{\intercal}: |s| \leq \beta \right\} \cdot \left\{a_t: |t| \leq \beta \right\}$ and $u_s^{\intercal}$ is the transpose of $u_s$. For any such $E_t$ as above, and any $(X,\omega) \in \Omega_1\mathcal{M}_g(2)$, the map $E_t \rightarrow \Omega_1\mathcal{M}_g(2)$ given by $h \mapsto h\cdot (X,\omega)$ is well-defined, since any $h \in E_t$ is in particular an element of \ $\mathrm{SL}(2,\mathbb{R})$. It is is injective when, for any lift of $(X,\omega)$, $(\tilde{X},\tilde{\omega})$, the restricted projection map
$\pi|_{B_r(\tilde{X},\tilde{\omega}) \cap E_t \cdot (\tilde{X},\tilde{\omega}) } : \Omega_1\mathcal{T}_g(2) \rightarrow \Omega_1\mathcal{M}_g(2)$ is injective for any $r >0$. To make precise the definition of a "small dimension transverse direction", we need a notion of transverse energy and a Margulis function. Let $t \geq 0$, and $(X,\omega) \in \Omega_1\mathcal{M}_g(2)$, and assume that the map  $E_t \rightarrow \Omega_1\mathcal{M}_g(2)$ described above is injective. For every $z \in E_{t} \cdot (X,\omega)$, set

\begin{displaymath}
I_{bal,t}(z):= \left\{w \in T_{z}^{bal}: 0 < ||w|| < r(z) 
 \ \textrm{and} \ \mathrm{exp}_z(w) \in E_t\cdot (X,\omega) \right\}
\end{displaymath}
where $T_{z}^{bal}$ is the balanced tangent space at $z$, $\mathrm{exp}$ is the exponential map in Lemma $\ref{exponential_map}$ and $0 < r(z) < 1/C_1$ is defined in the following way. Suppose $\tilde{r}(z) > 0$ is the maximum $r$ so that the restriction of the projection map from the Teichm{\"u}ller space to moduli space $\pi|_{B_{r}(\tilde{X},\tilde{\omega}) } : \Omega_1\mathcal{T}_g(2) \rightarrow \Omega_1\mathcal{M}_g(2)$ is injective. Let $0< \theta < 0.01$ and $B^{\mathrm{SL}(2,\mathbb{R})}_{\theta} := \left\{u_s^{\intercal}: |s| \leq \theta \right\} \cdot \left\{a_t: |t| \leq \theta \right\} \cdot \left\{u_s: |s| \leq \theta \right\}$. Assume $\theta(z)$ is the supremum over all $\theta$ so that the following holds; if $z' = h \cdot \mathrm{exp}_z(w)$, and $h \in B^{\mathrm{SL}(2,\mathbb{R})}_{\theta}$ is a non-identity element, there exists no $w' \in T_{z}^{bal}$ such that $\mathrm{exp}_z(w') = z'$. Define 

\begin{displaymath}
r(z):= \textrm{min} \left\{\theta(z), \tilde{r}(z), \mathrm{inj}(z) \right\}
\end{displaymath}
where $\mathrm{inj}(z)$ is defined as in Section \ref{Section3}. Note that for a fixed $z$, $I_{bal}(z)$ is a finite subset of $T_{\omega}^{bal}$, by definition and because $E_t$ is bounded. Let $0 < \nu < 1 $, and define the following Margulis function $f_{t}: E_{t}\cdot (X,\omega) \rightarrow (0,\infty)$ by 

\begin{equation}
f_{t}(z) =
\begin{cases}

\sum_{0 \neq w \in I_{bal}(z)} ||w||^{-\nu} & \text{if} \ I_{bal}(z) \neq 0 \\
{r(z)}^{-\nu} & otherwise. 

\end{cases}
\end{equation}
We would like a quantitative bound on the number of elements of $I_{bal,t}(z)$ in terms of $t$. It is more delicate to obtain such a bound in our setting, due to the inhomogeneous nature of moduli space. We require the following Lemma due to Avila-Gou{\"e}zel-Yoccoz.

\begin{lemma}[Lemma 5.2, \cite{AvilaExp}]
\label{AGY_pushforward}
Let $x=(X,\omega) \in \Omega_1\mathcal{M}_g(\sigma)$. For $v \in H^1(X,Z(\omega),\mathbb{C})$ a tangent vector, $t \geq 0$, and $s \in [0,1]$, we have the following inequalities

\begin{equation}
\label{AGY_inequality}
e^{-2-2t}||v|| \leq ||(a_tu_s)_{*}v|| \leq e^{2+2t}||v||
\end{equation}
where $(a_tu_s)_{*}v$ is the pushforward of the vector $v$ under the $a_tu_s$ action.
\end{lemma}

\begin{lemma}
\label{sheet_estimate}
Let $z \in E_t \cdot (X,\omega) \in \Omega_1\mathcal{M}_2(2)$, and assume $(X,\omega) \in \Omega_1\mathcal{M}_c$. That is, that $\mathrm{inj}(X,\omega) \geq c$. Then, $\#I_{bal,t}(z) \ll e^{6\kappa_4 t}$ for some $\kappa_4$ depending only on the stratum.
\end{lemma}

\begin{proof}
We first show that, for any $h \in E_t$, $r(hz) \gg e^{-\kappa_4 t}r(z)$ for some $\kappa_4$ depending only on our stratum $\Omega_1\mathcal{M}_2(2)$. Note, first, that since the $a_t$ action can only shorten a saddle connection by at most a factor of $e^{-t}$, $\mathrm{inj}(hz) \geq e^{-t}\mathrm{inj}(z)$, and so certainly $\mathrm{inj}(hz) \gg e^{-\kappa_4 t}\mathrm{inj}(z)$ for any $\kappa_4 > 1$. For the case of $\tilde{r}(z)$, we have by Proposition \ref{inj_radius_proposition} that $\tilde{r}(hz) \gg e^{-C_5\kappa_3 t} r(z)$, and consequently, treating the case of $\theta(hz)$ suffices.\\

We now show that $\theta(hz) \gg e^{-4t}\theta(z)$. Specifically, we will show that $\theta(hz) \geq \sqrt{2}/(4\sqrt{c\pi})e^{-2-2t}\theta(z)$. Unraveling the definition of $\theta(z)$, we know that for any $\delta>0$ there exist $w,w' \in T^{bal}_{z}$ so that $\mathrm{exp}_{z}(w) \in E_t \cdot (X,\omega)$ and $\mathrm{exp}_{z}(w') \in E_t \cdot (X,\omega)$, with $\mathrm{exp}_{z}(w') = h \cdot \mathrm{exp}_{z}(w)$ and $h \in B^{\mathrm{SL}(2,\mathbb{R})}_{\theta(z)-\delta}$. Recall that the decomposition of the tangent bundle into balanced and standard subbundles is $\mathrm{SL}(2,\mathbb{R})$-invariant. Therefore, $h_{*}(w)$ and $h_{*}(w') \in T^{bal}_{hz}$. We may assume $\mathrm{exp}_{hz}(h_{*}w)$ and $\mathrm{exp}_{hz}(h_{*}w')$ are well-defined inside a local period coordinate chart; if they are not, then $\theta(hz) > \tilde{r}(hz) \gg e^{-C_5\kappa_3 t} \tilde{r}(z) > e^{-C_5\kappa_3 t}r(z)$, and by the definition of $r(hz):= \mathrm{min} \left\{\theta(hz),\tilde{r}(hz),\mathrm{inj}(hz)\right\}$, $r(hz) \gg e^{-C_5\kappa_3 t}r(z)$. Taking $\kappa_4 = \\ \mathrm{max}\left\{1,C_5\kappa_3\right\}$, we are done. \\

Since $h_{*}(w)$ and $h_{*}(w') \in T^{bal}_{hz}$, therefore are transversal to the $\mathrm{SL}(2,\mathbb{R})$-orbit of $(X,\omega)$, we deduce $\mathrm{exp}_{hz}(h_{*}w)$ and $\mathrm{exp}_{hz}(h_{*}w')$ intersect the orbit $E_t\cdot(X,\omega)$. This is an affine subspace in this local period coordinate chart. The affine geodesic (straight line in local period coordinates) connecting $\mathrm{exp}_{hz}(h_{*}w)$ and $\mathrm{exp}_{hz}(h_{*}w')$ lies inside this affine subspace, and altogether we have $\mathrm{exp}_{hz}(h_{*}w) = h'\cdot \mathrm{exp}_{hz}(h_{*}w')$ with $h' \in E_t$. We would like an upper bound on $||h_{*}(w') - h_{*}(w)||$ in terms of the quantity $\theta(hz)(4/\pi\sqrt{2})$. By Equation $\ref{norm_comparison_equation}$, we have $d_{\mathcal{T}}(\mathrm{exp}_{h\cdot z}(h_{*}w'), \mathrm{exp}_{h\cdot z}(h_{*}w)) \geq (c/4\sqrt{2})||h_{*}(w') - h_{*}(w)||$, and
so 

\begin{align}
\label{AGY_theta}
||h_{*}(w') - h_{*}(w)|| \leq e^{t}(4\sqrt{2}/c)\theta(hz).
\end{align}
We now apply Lemma \ref{AGY_pushforward} to see that

\begin{align}
\label{contraction_vector_difference}
||h_{*}(w')-h_{*}(w)|| &= ||h_{*}(w'-w)|| \nonumber \\
&\geq e^{-2-2t}||w'-w||.
\end{align}
We combine Equations \ref{AGY_theta} and \ref{contraction_vector_difference}, and a similar estimate as Equation $\ref{AGY_theta}$ for $||w'-w||$ to obtain
$\theta(hz) \geq e^{-2-4t}(4\sqrt{2}/c)(\sqrt{c\pi}/16)(\theta(z) -\delta))$. Choosing $\delta$ to be arbitrarily close to $0$, we obtain

\begin{align}
\theta(hz) \geq \sqrt{2}/(4\sqrt{c\pi})e^{-2-4t}\theta(z).
\end{align}
Taking $\kappa_4 = \mathrm{max}\left\{C_5\kappa_3,4\right\}$, we are done. Denote by $m_{\mathrm{SL}(2,\mathbb{R})}$ the Haar measure on $\mathrm{SL}(2,\mathbb{R})$. We observe that for any $z \in E_t \cdot (X,\omega)$ and $w \in I_{bal,t}(z)$, we have

\begin{displaymath}
B^{\mathrm{SL}(2,\mathbb{R})}_{\sqrt{2c}/(4\sqrt{\pi})e^{-2-\kappa_4 t}}\mathrm{exp}_z(w) \subset E_{t+}\cdot (X,\omega)
\end{displaymath}
where $E_{t+}:= B^{\mathrm{SL}(2,\mathbb{R})}_{3\sqrt{2c}/(4\sqrt{\pi})e^{-2-\kappa_4t}}\cdot E_t$. We now have, by the definition of $I_{bal,t}(z)$, that for distinct $w,w' \in I_{bal,t}(z)$

\begin{equation}
\label{size_of_sheets}
B^{\mathrm{SL}(2,\mathbb{R})}_{\sqrt{2c}/(4\sqrt{\pi})e^{-2-\kappa_4 t}}\mathrm{exp}_z(w) \cap B^{\mathrm{SL}(2,\mathbb{R})}_{\sqrt{2c}/(4\sqrt{\pi})e^{-2-\kappa_4 t}}\mathrm{exp}_z(w') = \emptyset.
\end{equation}
Since $m_{\mathrm{SL}(2,\mathbb{R})}(E_{t+}) \ll e^{2t}$ and  $m_{\mathrm{SL}(2,\mathbb{R})}( B^{\mathrm{SL}(2,\mathbb{R})}_{ce^{-2-\kappa_4 t}\pi\sqrt{2}/4}) \gg e^{-3\kappa_4 t}$, the claim is proven.

\end{proof}

The following elementary observation will be convenient for us. \\

\begin{prop}
\label{polynomial_bound}

 \ Let $c \in \mathbb{R}, D > 0, \eta > 0$. Assume $\mathrm{max} \left\{a_i \right\} \gg \eta^2$. The set of all $x \in \mathbb{R}$ such that

\begin{equation}
\label{poly_inequality}
| a_1 e^{-Dt} + a_2(x-c) - a_3 e^{Dt}(x-c)^2| \leq Ce^{(-D+1)t}
\end{equation}
has measure $\ll \eta^{-4}e^{(\frac{-D+1}{2})t}$.

\end{prop}

\begin{proof}[Proof of Theorem \ref{main_theorem}]
In the proceeding discussion, we take $\beta = e^{-\frac{1}{6\kappa_4 + 100}t}$. Lemma $\ref{Main_quant_divergence}$, guarantees the existence constants $\alpha, C_3, C_4 >0$ so that for any $0 < \epsilon, \beta < 1$ and $(X,\omega) \in \Omega_1\mathcal{M}_g(\sigma)$ the following holds. Let $I \subset [-10,10]$ be an interval and assume $|I|\geq \beta$.
Then,
\begin{equation}
\label{non_divergence_in_proof}
{\big|\left\{s \in I : \mathrm{inj}\left(a_t u_s (X,\omega)\right) < {\epsilon}\right\}\big|} < C_4  {\epsilon}^{\alpha}|I|,
\end{equation}
so long as $t \geq \big\lvert \mathrm{log}\left(\beta  \ \mathrm{inj}(X,\omega)\right)^{-1}\big\rvert + C_3$. We will assume $t \geq \mathrm{max}\left\{ \big\lvert \mathrm{log}\left(\beta  \ \mathrm{inj}(X,\omega)\right)^{-1}\big\rvert + C_3, \mathrm{log}(1/C(x)) \frac{1}{\epsilon_x} \right\}$ where $C(x)$ and $\epsilon_x$ depend only on $(X,\omega)$ as per the discussion in Subsection $\ref{Norm_comparisons_closing}$. We fix such an interval $I$ that contains the subinterval $[0,1]$ and assume $t$ satisfies the above lower bound for the remainder of the proof. Let $s' \in [0,1]$ be such that $(X',\omega') = a_tu_{s'}(X,\omega) \in \Omega_1\mathcal{M}_{\beta}$, using the notation of Definition $\ref{inj_radius_def}$. We will show that if option $2$ in Theorem $\ref{main_theorem}$ fails, then the following holds. For every $(X',\omega')$ there exists an interval $I(\omega') \subset [0,1]$ with $\big|[0,1] \big \backslash I(\omega') \big| \leq 2C_3\beta^{\alpha}$ so that for all $s\in I(\omega')$

\begin{enumerate}[(a)]
\item $a_{7t}u_s(X',\omega') \in \Omega_1\mathcal{M}_{\beta}$,
\item $h \rightarrow a_{7t}u_s(X',\omega')$ is an injective map $E_t \rightarrow \Omega_1\mathcal{M}_2(2)$,
\item for all $z \in E_t\cdot a_{7t}u_s(X',\omega')$, we have $f_t(z) \leq e^{(6\kappa_4 + 100)t}$.
\end{enumerate}
Note that $u_sa_t = a_tu_{se^{-t}}$, so that 

\begin{align*}
a_{7t}u_sa_tu_{s'}(X,\omega) &=\\
&= a_{7t}a_tu_{se^{-t}}u_{s'}(X,\omega)\\ &=
a_{8t}u_{s' + se^{-t}}(X,\omega).
\end{align*}
which implies that option $1$ in Theorem $\ref{main_theorem}$ holds with $D=6\kappa_4 + 100$. Assume to the contrary that there exists $(X',\omega')$ and a subset $I'_{bad} \subset [0,1]$ with $\big|I'_{bad}\big| > 2C_3\beta^{\alpha}$ so that one of $(a), (b), (c)$ is false. By $\ref{non_divergence_in_proof}$, ${\big|\left\{s \in I : \mathrm{inj}\left(a_{7t} u_s (X',\omega')\right) < {\beta}\right\}\big|} < C_3  {\beta}^{\alpha}$, there exists $I_{bad} \subset [0,1]$ with $\big| I_{bad} \big| > 1 - C_3\beta^{\alpha} > 2C_3\beta^{\alpha} - C_3\beta^{\alpha} = C_3\beta^{\alpha}$ so that for all $s \in I_{bad}$, $a_{7t}u_s(X',\omega') \in \Omega_1\mathcal{M}_{\beta}$ but one of $(b)$ and $(c)$ is false. We will show that this implies option $2$ of Theorem $\ref{main_theorem}$.\\

We will prove the statement for option $(c)$ as option $(b)$ is similar. By definition of $I_{bad}$, we have that for any $s\in I_{bad}$, there exists a $z \in E_t\cdot a_{7t}u_s(X',\omega')$ so that $f_t(z) > e^{(6\kappa_4 +100)t}$. Since $a_{7t}u_s(X',\omega') \in \Omega_1\mathcal{M}_{\beta}$, we have, by the proof of Lemma $\ref{sheet_estimate}$, 

\begin{displaymath}
    r(ha_{7t}u_s(X',\omega')\gg \beta e^{-2t}
\end{displaymath}
for all $h\in E_t$. By the definition of $f_t$, if $I_{bal,t}(z) = \left\{0\right\}$, then $f_t(z) \ll \beta^{-1}e^{2t}$. Since we have assumed $t \geq \big\lvert \mathrm{log}\left(\beta  \ \mathrm{inj}(X,\omega)\right)^{-1}\big\rvert + C_3$, we have 

\begin{align*}
f_t(z) &\ll \beta^{-1}e^{2t}\\
&\leq \beta^{-5}e^{2C_3} \ll \beta^{-5}.
\end{align*}
If we take $t$ large enough, this contradicts $f_t(z)> e^{(6\kappa_4 + 100)t}$, so we may assume $I_{bal,t}(z) \neq \left\{0\right\}$. By Lemma $\ref{sheet_estimate}$, $\#I_{bal,t}(z) \ll e^{6\kappa_4 t}$. By the definition of $I_{bal,t}(z)$, there exists $w \in T_{z}^{bal}$ with $||w|| \ll e^{-100t}$. \\
Altogether, there exists some $\bar{z} \in E_t \cdot a_{7t}u_s(X',\omega')$ so that $\bar{z} = \mathrm{exp}_z(\omega)$ with $||w|| \ll e^{-100t}$, and for some $\mathbf{h}_2 \in E_t$ we have $\bar{z} = \mathbf{h}_2 a_{7t}u_s(X',\omega')$. Let $\mathbf{h}_1 \in E_t$ be such that $z = \mathbf{h}_1 a_{7t}u_s(X',\omega')$ and note that $\bar{z} = \mathbf{h}_2 \mathbf{h}^{-1}_1 z$. Henceforth, for such $\mathbf{h}_2$, $\mathbf{h}_2$, we will use the shorthand $\mathbf{h}_s:= \mathbf{h}_2 \mathbf{h}^{-1}_1$.\\

$\textbf{Distinct conjugacy classes of mapping class group elements}$ \\

The structure of the remainder of the proof is an adaptation of the proof of Proposition 6.1 in $\cite{lindenstrauss2022polynomial1}$ to moduli space. The proof is quite similar, though requires modifications significant enough to merit a detailed treatment here. \\

Since $|I_{bad}| > C_3\beta^{\alpha}$, one can find intervals $J, J' \subset [0,1]$ with $d(J,J') \gg \beta^{\alpha}$, $|J|, |J'| \gg \beta^{\alpha}$ and $|J \cap I_{bad}| \geq \beta$, $|J' \cap I_{bad}| \geq \beta$. Set $J_{\beta} = J \cap I_{bad}$. We will further use the shorthand $\mathbf{p}_s:=a_{7t}u_s$. Note that we have the following inequalities.

\begin{align*}
d_{\mathcal{T}}\left((X',\omega'), (\mathbf{p}_s)^{-1}\mathbf{h}_s\mathbf{p}_s(X',\omega')\right) &\ll e^{14t}d_{\mathcal{T}}\left(\mathbf{p}_s(X',\omega'), \mathbf{h}_s\mathbf{p}_s(X',\omega')\right) \\
&\ll e^{-84t}.
\end{align*}
Thus, there exist lifts $(\tilde{X}',\tilde{\omega}')$ of $(X',\omega')$ and, for each lift, an element ${\mathbf{\gamma}}_s$ in the mapping class group so that

\begin{equation}
\label{smallinequality}
d_{\mathcal{T}}\left((\tilde{X}',\tilde{\omega}'),(\mathbf{p}_s)^{-1}\mathbf{h}_s\mathbf{p}_s {\mathbf{\gamma}}_s(\tilde{X}',\tilde{\omega}') \right) \ll e^{-84t}
\end{equation}
using the fact the mapping class group commutes with $\mathrm{SL}(2,\mathbb{R})$. Suppose that $(\tilde{\tilde{X}}', \tilde{\tilde{\omega}}')$ is a different lift of $(X,\omega)$, $\gamma \cdot (\tilde{\tilde{X}}', \tilde{\tilde{\omega}}') = (\tilde{X}',\tilde{\omega}')$, for some $\gamma$. Then,

\begin{displaymath}
d_{\mathcal{T}}\left((\tilde{\tilde{X}}',\tilde{\tilde{\omega}}'), (\mathbf{p}_s)^{-1}\mathbf{h}_s\mathbf{p}_s \gamma {\mathbf{\gamma}}_s (\gamma)^{-1}(\tilde{X}',\tilde{\omega}')\right) = d_{\mathcal{T}}\left((\tilde{X}',\tilde{\omega}'),(\mathbf{p}_s)^{-1}\mathbf{h}_s\mathbf{p}_s {\mathbf{\gamma}}_s(\tilde{X}',\tilde{\omega}') \right).
\end{displaymath}
So, to each $s$, we associate a conjugacy class $[\mathbf{\gamma}_s]$. Suppose that for $s \neq s'$, $[\mathbf{\gamma}_s] = [\mathbf{\gamma}_{s'}]$. Then, for a lift $(\tilde{X}', \tilde{\omega}')$, we have

\begin{equation}
\label{gamma}
d_{\mathcal{T}}\left((\tilde{X}',\tilde{\omega}'),(\mathbf{p}_s)^{-1}\mathbf{h}_s\mathbf{p}_s {\mathbf{\gamma}}_s(\tilde{X}',\tilde{\omega}') \right) = d_{\mathcal{T}}\left((\tilde{X}',\tilde{\omega}'),(\mathbf{p}_{s'})^{-1}\mathbf{h}_{s'}\mathbf{p}_{s'} {\mathbf{\gamma}}_{s}(\tilde{X}',\tilde{\omega}') \right).
\end{equation}
By the triangle inequality, and Equation \ref{smallinequality}  we have

\begin{equation}
\label{fundamental_inequality}
d_{\mathcal{T}}\left((\mathbf{p}_{s'})^{-1}\mathbf{h}_{s'}\mathbf{p}_{s'} {\mathbf{\gamma}}_{s}(\tilde{X}',\tilde{\omega}'),(\mathbf{p}_s)^{-1}\mathbf{h}_s\mathbf{p}_s {\mathbf{\gamma}}_s(\tilde{X}',\tilde{\omega}') \right) \ll e^{-84t}.
\end{equation}
By the definition of Teichm{\"u}ller distance, and using the first order approximation $\mathrm{log}(x) \sim (x-1)$ for $x$ near one, we obtain from Equation $\ref{fundamental_inequality}$

\begin{equation}
\label{4.contradiction_inq}
||(\mathbf{p}_{s'})^{-1}\mathbf{h}_{s'}\mathbf{p}_{s'}(\mathbf{p}_{s})^{-1}\mathbf{h}_{s}\mathbf{p}_{s}|| \ll e^{-68t}.
\end{equation}
By repeated applications of the facts on the matrix norm in Subsection $\ref{subsection_matrix_norm}$, we have 

\begin{equation}
\label{4.contradiction_inq2}
\frac{||\mathbf{p}_{s'}(\mathbf{p}_s)^{-1} \mathbf{h}_{s} \mathbf{p}_{s} (\mathbf{p}_{s'})^{-1}||}{||\mathbf{h}_{s'}||} \ll e^{14t}||(\mathbf{p}_{s'})^{-1}\mathbf{h}_{s'}\mathbf{p}_{s'}(\mathbf{p}_{s})^{-1}\mathbf{h}_{s}\mathbf{p}_{s}||.
\end{equation}
By proper discontinuity of the action of the mapping class group, we may find $\chi > 0$ be such that for any non-identity (conjugacy class) $\gamma_s$ we have for $\mathbf{h}_s = \begin{bmatrix}
a_{1,s} & a_{2,s} \\
a_{3,s} & a_{4,s} 
\end{bmatrix}$,
where $|a_{i,s}| \leq 10e^{2t}$,

\begin{equation}
\label{matrix_inequality_1}
||(\mathbf{p}_s)^{-1}\mathbf{h}_s (\mathbf{p}_s) - I|| = \Bigg\lvert \Bigg\lvert u_{-s} \begin{bmatrix}
a_{1,s} & e^{-7t}a_{2,s} \\
e^{7t}a_{3,s} & a_{4,s}  
\end{bmatrix} u_{s} - I \Bigg\rvert \Bigg\rvert \geq 20 \beta^{2\alpha}\chi.
\end{equation} 
From this, we conclude $\mathrm{max} \left\{e^{7t}|a_{3,s}|, |a_1 - 1|, |a_4 - 1| \right\} \gg \beta^{2\alpha}$. Set $\tau = e^{7t}(s'-s)$. We compute

\begin{align}
\begin{split}
\label{matrix_computation}
u_{\tau}\mathbf{h}_s u_{-\tau} &= \mathbf{p}_{s'}(\mathbf{p}_s)^{-1} \mathbf{h}_{s} \mathbf{p}_{s} (\mathbf{p}_{s'})^{-1} \\
&= \begin{bmatrix}
a_{1,s} + a_{3,s}\tau & a_{2,s} + (a_{4,s} - a_{1,s})\tau - a_{3,s}\tau^2 \\
a_{3,s} & a_{4,s} - a_{3,s}\tau
\end{bmatrix}.
\end{split}
\end{align}
By Proposition \ref{polynomial_bound}, and the fact that $\mathrm{max} \left\{e^{7t}|a_{3,s}|, |a_1 - 1|, |a_4 - 1| \right\} \gg \beta^{\alpha}$, we may conclude that for every $s \in J_{\beta}$, the set of $s' \in J_{\beta}$ so that

\begin{equation}
\label{polynomial_bound_separation}
|a_{2,s}e^{-7t} + (a_4 - a_1)(s'-s) - a_3e^{7t}(s'-s)^2| \leq e^{-6t}
\end{equation}
is of measure $\ll \beta^{-2\alpha}e^{-3t}$. If we denote by $J_{\beta,s}$ the set of $s' \in J_{\beta}$ for which Equation $\ref{polynomial_bound_separation}$ is valid, and we let $s' \in J_{\beta} \setminus J_{\beta,s}$ then for all such $s'$ we compute

\begin{displaymath}
\frac{||u_\tau \mathbf{h}_s u_{-\tau}||}{||\mathbf{h}_{s'}||} \gg e^{-2t}
\end{displaymath}
using Equation $\ref{matrix_computation}$ and the fact that $||\mathbf{h}_{s'}|| \ll e^{2t}$. This is in contradiction to  Equations $\ref{4.contradiction_inq}$ and $\ref{4.contradiction_inq2}$. We conclude that for $s \in J_{\beta}$, the set of $s' \in J_{\beta}$ for which $[\gamma_s] = [\gamma_{s'}]$ has measure $\ll \beta^{-2\alpha}e^{-3t}$. Therefore, the set of $s' \in J_{\beta}$ for which $(\mathbf{p}_s)^{-1} \mathbf{h}_{s} \mathbf{p}_s = (\mathbf{p}_{s'})^{-1} \mathbf{h}_{s'} \mathbf{p}_{s'}$ is of measure $\ll \beta^{-2\alpha}e^{-3t}$. Since $||(\mathbf{p}_s)^{-1} \mathbf{h}_{s} \mathbf{p}_s|| \ll e^{9t}$, the set

\begin{displaymath}
\left\{ (\mathbf{p}_s)^{-1} \mathbf{h}_{s} \mathbf{p}_s, \ s \in J_{\beta} \right\}
\end{displaymath}
has cardinality $\gg e^{2.9t}$. \\

\textbf{Existence of hyperbolic element} \\

We would like to show the group $\mathbf{G}$ generated by 
$\langle (\mathbf{p}_s)^{-1} \mathbf{h}_{s} \mathbf{p}_s, \ s \in I_{bad} \rangle$ cannot be unipotent; proceeding by contradiction, assume that 

\begin{equation}
\label{unipotent_conjugation}
(\mathbf{p}_s)^{-1} \mathbf{h}_s \mathbf{p}_s = u_{-s} \begin{bmatrix}
a_{1,s} & e^{-7t}a_{2,s} \\
e^{7t}a_{3,s} & a_{4,s}  
\end{bmatrix} u_{s} \in gUg^{-1}
\end{equation}
for all $s \in I_{bad}$, where $U = \left\{\begin{bmatrix}
1 & r \\
0 & 1  
\end{bmatrix}, \, r \in \mathbb{R}\right\}$ and $g \in \mathrm{SL}(2,\mathbb{R})$. Since $\mathbf{G}$ is unipotent, we know

\begin{displaymath}
\big\lvert \left\{ \mathbf{p}_s)^{-1} \mathbf{h}_s \mathbf{p}_s: \ ||\mathbf{p}_s)^{-1} \mathbf{h}_s \mathbf{p}_s|| \leq e^{4t/3}\right\}\big\rvert \ll e^{8t/3}.
\end{displaymath}
We have just shown $\big\lvert \left\{ (\mathbf{p}_s)^{-1} \mathbf{h}_{s} \mathbf{p}_s, \ s \in J_{\beta} \right\} \big\rvert \gg e^{2.9t}$ so we can therefore conclude

\begin{equation}
\label{gamma_bound}
\big\lvert \left\{ \mathbf{p}_s)^{-1} \mathbf{h}_s \mathbf{p}_s: \ ||\mathbf{p}_s)^{-1} \mathbf{h}_s \mathbf{p}_s|| \geq e^{4t/3} \ \textrm{and} \ s \in J_{\beta} \right\} \big\rvert \gg e^{t/4}.
\end{equation}
Let $\beta^{2\alpha}\chi \leq \Psi \leq e^{4t/3}$. A computation shows that if 

\begin{displaymath}
||(\mathbf{p}_s)^{-1}\mathbf{h}_s (\mathbf{p}_s) - I|| = \Bigg\lvert \Bigg\lvert u_{-s} \begin{bmatrix}
a_{1,s} & e^{-7t}a_{2,s} \\
e^{7t}a_{3,s} & a_{4,s}  
\end{bmatrix} u_{s} - I \Bigg\rvert \Bigg\rvert \geq 20  \Psi
\end{displaymath}
then, for $t$ large enough 
\begin{equation}
\label{a3_bound}
|a_{3,s}| \geq \Psi e^{-7t}.
\end{equation}
By Equation $\ref{unipotent_conjugation}$, for each $s \in I_{bad}$, there is $r_0 \in \mathbb{R}$ so that

\begin{equation}
\label{first_Matrix_conj}
u_{-s}\begin{bmatrix}
a_{1,s} & e^{-7t}a_{2,s} \\
e^{7t}a_{3,s} & a_{4,s}
\end{bmatrix}u_{s}= g\begin{bmatrix}
1 & r_0 \\
0 & 1
\end{bmatrix}g^{-1}. 
\end{equation}
Let $s_0 \in J' \cap I_{bad}$. Setting $u_{s_0}g = \begin{bmatrix}
a & b \\
c & d
\end{bmatrix}$, Equation $\ref{first_Matrix_conj}$ gives

\begin{equation}
\label{second_Matrix_conj}
\begin{bmatrix}
a_{1,s_{0}} & e^{-7t}a_{2,s_{0}} \\
e^{7t}a_{3,s_{0}} & a_{4,s_{0}}
\end{bmatrix} = \begin{bmatrix}
1-acr_0 & a^2r_0 \\
-c^2r_0 & 1+acr_0
\end{bmatrix} 
\end{equation}
for some $r_0 \in \mathbb{R}$. Applying Equation $\ref{a3_bound}$ with $\Psi = \beta^{2\alpha}\chi$, we have $|a_{3,s_{0}}| \geq e^{-7t}\beta^{2\alpha}\chi$, and, since $g$ is fixed so its entries are bounded by absolute constants, we may compare matrix entries on the right and left hand sides to conclude $|a| \ll e^{-2.9t}$ and $|c| \gg 1$. Now, let $s \in J_{\beta}$. By Equation $\ref{gamma_bound}$, we may assume $||(\mathbf{p}_s)^{-1}\mathbf{h}_s\mathbf{p}_s|| \geq e^{4t/3}$. Set $s_1 = s - s_0$, $a'_{2,s} = e^{-7t}a_{2,s}$, $a'_{3,s} = e^{7t}a_{3,s}$. By Equation $\ref{a3_bound}$ applied with $\Psi = e^{4t/3}$, we deduce $|a'_{3,s}| \geq e^{4t/3}$. For the sake of once again comparing matrix entries, apply Equation $\ref{unipotent_conjugation}$ to obtain 

\begin{align*}
u_{-s_{1}}
\begin{bmatrix}
a_{1,s} & a'_{2,s} \\
a'_{3,s} & a_{4,s}
\end{bmatrix}u_{s_{1}} &= \begin{bmatrix}
a_{1,s} - s_1a'_{3,s} & a'_{2,s} + (a_{4,s} - a_{1,s})s_1 - a'_{3,s}(s_1)^2 \\
a'_{3,s} & a_{4,s} + s_1a'_{3,s}
\end{bmatrix} \\
 &=  \begin{bmatrix}
1-acr & a^2r \\
-c^2r & 1+acr
\end{bmatrix}
\end{align*}
where $r \in \mathbb{R}$. Compare matrix entries using $|a'_{3,s}| \geq e^{4t/3}$ and, by $d(J,J') \gg \beta^{\alpha}$, the inequality $\beta \leq \beta^{\alpha} \leq s_1 \leq 1$. We conclude $|a|^2$ and $|c|^2$ are comparable in size. This contradicts $|a| \ll e^{-2.9t}$ and $|c| \gg 1$ for $t$ large. \\

In light of the discussion in Subsection $\ref{transversal_separation}$, we may assume $(\mathbf{p}_s)^{-1}\mathbf{h}_s\mathbf{p}_s$ are not elliptic, and we have just demonstrated $\langle (\mathbf{p}_s)^{-1}\mathbf{h}_s\mathbf{p}_s, \ s \in I_{bad}\rangle$ is not unipotent. It is, in particular not generated by a single parabolic element. It must therefore contain two transverse parabolic elements, and by the ping-pong lemma must contain a hyperbolic element. Considering traces, there must be an $\mathbf{h}_s$ hyperbolic. \\ 

All in all, there exist $z,\bar{z} \in E_t \cdot a_{7t}u_s \cdot (X',\omega')$, $\bar{z} = \mathbf{h}_s\cdot z$ where $d_{\mathcal{T}}(z,\bar{z}) \ll e^{-80t}$.  By Theorem \ref{Hodge_Teich_norm_comparison}, $d_H(z,\bar{z}) \ll e^{-60t}$. By Theorem \ref{Anosov_closing_lemma}, there exists a point $z'$ fixed by $\mathbf{h}_s$ with $d_{\mathcal{T}}(z',z) \ll e^{-60t}$. In particular, by Equation $\ref{size_of_sheets}$, and the fact that $d_{\mathcal{T}}(z',z) \ll e^{-60t}$, $\gamma_s$ corresponding to $\mathbf{h}_s$ must be non-trivial. The translation length $l(\mathbf{h}_s)$ satisfies $e^{-t/2} \ll \mathrm{log}(l(\mathbf{h}_s)) \ll e^{20t}$. The diagonalizable element $\mathbf{h}_s$ is in the Veech group of $z'$ and by Theorem $\ref{Veech_hyperbolic}$, $z'$ lies on a Teichm{\"u}ller curve. Theorem $\ref{main_theorem}$ follows from this, taking $D \geq 6\kappa_4 + 100$.
\end{proof}

\bibliographystyle{plain}
\bibliography{bibliography.bib}

\begin{thebibliography}{10}

\bibitem{article}
V.~Arnold, M.~Atiyah, P.~Lax, and B.~Mazur.
\newblock Mathematics: Frontiers and perspectives.
\newblock {\em The American Mathematical Monthly}, 108, 04 2001.

\bibitem{article2}
Masayuki Asaoka and Kei Irie.
\newblock A ${C}^\infty$ closing lemma for {H}amiltonian diffeomorphisms of closed surfaces.
\newblock {\em Geometric and Functional Analysis}, 26, 12 2015.

\bibitem{athreya2006quantitative}
Jayadev~S Athreya.
\newblock Quantitative recurrence and large deviations for {T}eichm{\"u}ller geodesic flow.
\newblock {\em Geometriae Dedicata}, 119:121--140, 2006.

\bibitem{AvilaExp}
A.~Avila, S.~Gouëzel, and J.C. Yoccoz.
\newblock Exponential mixing for the {T}eichmüller flow.
\newblock {\em Publ.math.IHES}, 104:143--211, 2006.

\bibitem{Avila2010SmallEO}
Artur Avila and Sebastien Gouezel.
\newblock Small eigenvalues of the {L}aplacian for algebraic measures in moduli space, and mixing properties of the {T}eichm{\"u}ller flow.
\newblock {\em Annals of Mathematics}, 178:385--442, 2010.

\bibitem{bourgain2011stationary}
Jean Bourgain, Alex Furman, Elon Lindenstrauss, and Shahar Mozes.
\newblock Stationary measures and equidistribution for orbits of nonabelian semigroups on the torus.
\newblock {\em Journal of the American Mathematical Society}, 24(1):231--280, 2011.

\bibitem{chaika2023space}
Jon Chaika, Osama Khalil, and John Smillie.
\newblock On the space of ergodic measures for the horocycle flow on strata of abelian differentials.
\newblock {\em arXiv preprint arXiv:2104.00554}, 2023.

\bibitem{Chaika2020TremorsAH}
Jon Chaika, Jonathan Smillie, and Barak Weiss.
\newblock Tremors and horocycle dynamics on the moduli space of translation surfaces.
\newblock {\em arXiv preprint arXiv:2004.04027}, 2022.

\bibitem{einsiedler2009effective}
Manfred Einsiedler, Gregory Margulis, and Akshay Venkatesh.
\newblock Effective equidistribution for closed orbits of semisimple groups on homogeneous spaces.
\newblock {\em Inventiones mathematicae}, 177(1):137--212, 2009.

\bibitem{Eskin2001AsymptoticFO}
A.~V. Eskin and Howard~A. Masur.
\newblock Asymptotic formulas on flat surfaces.
\newblock {\em Ergodic Theory and Dynamical Systems}, 21:443 -- 478, 2001.

\bibitem{eskin2018invariant}
Alex Eskin and Maryam Mirzakhani.
\newblock Invariant and stationary measures for the {SL}(2,$\mathbb{R}$) action on moduli space.
\newblock {\em Publications math{\'e}matiques de l'IH{\'E}S}, 127(1):95--324, 2018.

\bibitem{eskin2015isolation}
Alex Eskin, Maryam Mirzakhani, and Amir Mohammadi.
\newblock Isolation, equidistribution, and orbit closures for the {SL}(2, $\mathbb{R}$) action on moduli space.
\newblock {\em Annals of Mathematics}, pages 673--721, 2015.

\bibitem{EMM_effective_simple}
Alex Eskin, Maryam Mirzakhani, and Amir Mohammadi.
\newblock Effective counting of simple closed geodesics on hyperbolic surfaces.
\newblock {\em Journal of the European Mathematical Society}, 24, 10 2021.

\bibitem{eskin2019counting}
Alex Eskin, Maryam Mirzakhani, and Kasra Rafi.
\newblock Counting closed geodesics in strata.
\newblock {\em Inventiones mathematicae}, 215(2):535--607, 2019.

\bibitem{Filip2014ZeroLE}
Simion Filip.
\newblock Zero lyapunov exponents and monodromy of the {K}ontsevich–{Z}orich cocycle.
\newblock {\em Duke Mathematical Journal}, 166:657--706, 2014.

\bibitem{filip2016splitting}
Simion Filip.
\newblock Splitting mixed {H}odge structures over affine invariant manifolds.
\newblock {\em Annals of Mathematics}, 183(2):681--713, 2016.

\bibitem{forni2002deviation}
Giovanni Forni.
\newblock Deviation of ergodic averages for area-preserving flows on surfaces of higher genus.
\newblock {\em Annals of Mathematics}, 155(1):1--103, 2002.

\bibitem{Furstenberg1}
H.~Furstenberg.
\newblock Disjointness in ergodic theory, minimal sets, and a problem in {D}iophantine approximation.
\newblock {\em Math. Systems Theory 1}, pages 1--49, 1967.

\bibitem{green2012quantitative}
Ben Green and Terence Tao.
\newblock The quantitative behaviour of polynomial orbits on nilmanifolds.
\newblock {\em Annals of Mathematics}, pages 465--540, 2012.

\bibitem{article3}
Ursula Hamenstaedt.
\newblock Dynamics of the {T}eichmüller flow on compact invariant sets.
\newblock {\em Journal of Modern Dynamics}, 4, 05 2007.

\bibitem{Bowen_T}
Ursula Hamenstaedt.
\newblock Bowen's construction for the teichmüller flow.
\newblock {\em Journal of Modern Dynamics}, 7, 07 2010.

\bibitem{Veech_groups}
Pascal Hubert and Thomas~A. Schmidt.
\newblock {Infinitely generated Veech groups}.
\newblock {\em Duke Mathematical Journal}, 123(1):49 -- 69, 2004.

\bibitem{hutchings2024elementary}
Michael Hutchings.
\newblock Elementary spectral invariants and quantitative closing lemmas for contact three-manifolds.
\newblock {\em arXiv preprint arXiv:2208.01767}, 2024.

\bibitem{Hodge_Teich}
Jeremy Kahn and Alex Wright.
\newblock Hodge and teichmüller.
\newblock {\em Journal of Modern Dynamics}, 18:149, 01 2022.

\bibitem{katok1995introduction}
A.~Katok, A.B. Katok, and B.~Hasselblatt.
\newblock {\em Introduction to the Modern Theory of Dynamical Systems}.
\newblock Encyclopedia of Mathematics and its Applications. Cambridge University Press, 1995.

\bibitem{lindenstrauss2022polynomial1}
E.~Lindenstrauss and A.~Mohammadi.
\newblock Polynomial effective density in quotients of $\mathbb{H}^3$ and $\mathbb{H}^2\times \mathbb{H}^2$.
\newblock {\em Inventiones mathematicae}, 231, 09 2022.

\bibitem{lindenstrauss2014effective}
Elon Lindenstrauss and Gregory Margulis.
\newblock Effective estimates on indefinite ternary forms.
\newblock {\em Israel Journal of Mathematics}, 203(1):445--499, 2014.

\bibitem{lindenstrauss2022effectiveunipotent}
Elon Lindenstrauss, Amir Mohammadi, and Zhiren Wang.
\newblock Effective equidistribution for some one parameter unipotent flows.
\newblock {\em arXiv preprint arXiv:2211.11099}, 2022.

\bibitem{lindenstrauss2022polynomial}
Elon Lindenstrauss, Amir Mohammadi, and Zhiren Wang.
\newblock Polynomial effective equidistribution.
\newblock {\em arXiv preprint arXiv:2202.11815}, 2022.

\bibitem{lindenstrauss2023effective}
Elon Lindenstrauss, Amir Mohammadi, Zhiren Wang, and Lei Yang.
\newblock An effective version of the {O}ppenheim conjecture with a polynomial error rate.
\newblock {\em arXiv preprint arXiv:2305.18271}, 2023.

\bibitem{mcmullen2003billiards}
Curtis McMullen.
\newblock {Billiards and Teichm{\"u}ller curves on Hilbert modular surfaces}.
\newblock {\em Journal of the American Mathematical Society}, 16(4):857--885, 2003.

\bibitem{McMullen_hyperbolic}
Curtis McMullen.
\newblock Dynamics of $\mathrm{SL}(2,\mathbb{R})$ over moduli space in genus two.
\newblock {\em Annals of mathematics, ISSN 0003-486X, Vol. 165, Nº 2, 2007, pags. 397-456}, 165, 03 2007.

\bibitem{infinite_complexity_McMullen}
Curtis~T. McMullen.
\newblock Teichmüller geodesics of infinite complexity.
\newblock {\em Acta Mathematica}, 191(2):191 -- 223, 2007.

\bibitem{minsky2002nondivergence}
Yair Minsky and Barak Weiss.
\newblock Nondivergence of horocyclic flows on moduli space.
\newblock {\em Journal fur die Reine und Angewandte Mathematik}, (552):131--177, 2002.

\bibitem{fc49ac35-5adf-3bc9-b971-cec3f7c4e0a5}
Charles~C. Pugh.
\newblock An improved closing lemma and a general density theorem.
\newblock {\em American Journal of Mathematics}, 89(4):1010--1021, 1967.

\bibitem{Ratner4}
Marina Ratner.
\newblock Rigidity of time changes for horocycle flows.
\newblock {\em Acta Mathematica}, 156:1--32, 1986.

\bibitem{Ratner2}
Marina Ratner.
\newblock On measure rigidity of unipotent subgroups of semisimple groups.
\newblock {\em Acta Mathematica}, 165(3-4):229--309, 1990.

\bibitem{Ratner}
Marina Ratner.
\newblock On {R}aghunathan's measure conjecture.
\newblock {\em Annals of Mathematics}, 134(3):545--607, 1991.

\bibitem{Ratner3}
Marina Ratner.
\newblock Raghunathan’s topological conjecture and distributions of unipotent flows.
\newblock {\em Duke Mathematical Journal}, 63:235--280, 1991.

\bibitem{sanchez2023effective}
Anthony Sanchez.
\newblock Effective equidistribution of large dimensional measures on affine invariant submanifolds.
\newblock {\em arXiv preprint arXiv:2306.06740}, 2023.

\bibitem{smillie2023horospherical}
John Smillie, Peter Smillie, Barak Weiss, and Florent Ygouf.
\newblock Horospherical dynamics in invariant subvarieties.
\newblock {\em arXiv preprint arXiv:2303.07188}, 2023.

\bibitem{smillie2004minimal}
John Smillie and Barak Weiss.
\newblock Minimal sets for flows on moduli space.
\newblock {\em Israel Journal of Mathematics}, 142(1):249--260, 2004.

\bibitem{solan2024critical}
Omri~Nisan Solan.
\newblock Critical exponent gap and leafwise dimension.
\newblock {\em arXiv preprint arXiv:2404.00700}, 2024.

\bibitem{strombergsson2015effective}
Andreas Str{\"o}mbergsson.
\newblock An effective {R}atner equidistribution result for {SL}(2, $\mathbb{R}$)$\ltimes$ $\mathbb{R}^2$.
\newblock {\em Duke Mathematical Journal}, 164(5):843--902, 2015.

\bibitem{SU2024108839}
Weixu Su and Shenxing Zhang.
\newblock A comparison between {A}vila-{G}ouëzel-{Y}occoz norm and {T}eichmüller norm.
\newblock {\em Topology and its Applications}, 345:108839, 2024.

\bibitem{wright2014field}
Alex Wright.
\newblock The field of definition of affine invariant submanifolds of the moduli space of {A}belian differentials.
\newblock {\em Geometry \& Topology}, 18(3):1323--1341, 2014.

\bibitem{Yang}
Lei Yang.
\newblock Effective version of {R}atner's equidistribution theorem for $\mathrm{SL}(3,\mathbb{R})$.
\newblock {\em arXiv preprint arXiv:2208.02525}, 2023.

\end{thebibliography}

\end{document}